\documentclass[11pt]{article}
\usepackage[utf8]{inputenc}
\usepackage[margin=2.5cm]{geometry}
\usepackage{setspace}
\usepackage{amssymb,amsfonts,amsmath,amsthm,bbm}
\usepackage{mathtools}
\usepackage{comment}
\usepackage{enumerate}
\usepackage{xcolor}

\title{Majority dynamics on random graphs: the multiple states case}
\author{Jordan Chellig \hspace{1cm} Nikolaos Fountoulakis\footnote{Email for correspondence: n.fountoulakis@bham.ac.uk} \\ \\ 
School of Mathematics \\
University of Birmingham \\
Edgbaston, B15 2TT \\
United Kingdom}
\date{\today}

\newtheorem{theorem}{Theorem}[section]

\newtheorem{lemma}[theorem]{Lemma}

\newtheorem{claim}[theorem]{Claim}
\newtheorem{corollary}[theorem]{Corollary}
\newtheorem{remark}[theorem]{Remark}

\newcommand{\bin}{\mathrm{Bin}}
\newcommand{\hyp}{\mathrm{Hyper}}
\newcommand{\prb}[1]{\mathbb{P}\left(#1 \right)}
\newcommand{\E}[1]{\mathbb{E} \left[#1 \right]}
\newcommand{\Es}[1]{\mathbb{E}_{\sim} \left[#1 \right]}
\newcommand{\Ens}[1]{\mathbb{E}_{\not\sim} \left[#1 \right]}
\newcommand{\ind}[1]{1_{#1}}

\DeclareMathOperator*{\argmax}{argmax}
\DeclareMathOperator{\mult}{Mult}

\newcommand{\eps}{\varepsilon}

\newcommand{\Ncal}{\mathcal{N}}
\newcommand{\Pcal}{\mathcal{P}}
\newcommand{\Scal}{\mathcal{S}}
\newcommand{\Lcal}{\mathcal{L}}
\newcommand{\Kcal}{\mathcal{K}}
\newcommand{\Bcal}{\mathcal{B}}
\newcommand{\Rcal}{\mathcal{R}}
\newcommand{\Dcal}{\mathcal{D}}
\newcommand{\Mcal}{\mathcal{M}}
\newcommand{\Ev}{\mathcal{E}}
\newcommand{\x}{\boldsymbol{x}}
\newcommand{\y}{\boldsymbol{y}}

\newcommand{\q}{\boldsymbol{q}}
\newcommand{\h}{\boldsymbol{h}}
\newcommand{\cv}{\boldsymbol{c}}
\newcommand{\e}{\boldsymbol{e}}
\newcommand{\un}{\mathbf{\mathrm{I}}}

\begin{document}

\maketitle

\begin{abstract}
We study the evolution of majority dynamics with more than two states on the binomial random graph $G(n,p)$. 
In this process, each vertex has a state in $\{1,\ldots, k\}$, with $k\geq 3$,  
and at each round every vertex adopts state $i$ if it has more neighbours in state $i$ than in any other state. Ties are resolved randomly. 
We show that with high probability the process reaches unanimity in  at most three  rounds, if $np\gg n^{2/3}$. 
\end{abstract}
\noindent

\section{Majority dynamics on a graph}

In this paper, we will study the \emph{synchronous} version of majority dynamics on a graph $G=(V,E)$.  
After each round, every vertex $v\in V$ will be in one of the states in $\Scal_k:= \{1,\ldots, k\}$, for an integer $k\geq 2$.  
 For $t\geq 1$, let $S_t : V \to \Scal_k$ be the  
\emph{state} of the vertices after the completion of the $t$th round, whereas $S_0 : V \to \Scal_k$ denotes the initial state. 
The version of the process we will analyse evolves as follows. Let $t\geq 1$ and for each vertex $v\in V$ and 
$i \in \Scal_k$, let $n_i(v;t-1) =\left| N_{G}(v) \cap \{u : S_{t-1}(u) = i \} \right|$; this is the number of neighbours 
of $v$ in $G$ which are at state $i$ after round $t-1$.  
Then $S_t(v)$ is selected uniformly at random from the set $\argmax \{ n_{i} (v;t-1)\}_{i\in \Scal_k}$. In particular, 
if there is  $i \in \Scal_k$ such that $n_{i} (v;t-1) > n_{j} (v;t-1)$ for all $j \not = i$, that is, 
$\argmax \{ n_{j} (v;t-1)\}_{j\in \Scal_k} = \{i\}$, then $S_t (v)= i$.

Models of this flavour were introduced in other fields several decades ago. 
A prominent example is a model that was introduced by M. Granovetter~\cite{ar:Granno1978}, aiming to 
explain the emergence of collective behaviour through the interaction of individual preferences and their aggregation. More specifically, the aim was to explain social phenomena such as the diffusion of innovations as well as the spread of rumours and diseases. In this setting the choice is among two states, \emph{active} or \emph{inactive}, and each vertex in the network has a different \emph{activation threshold}. 
In particular, a vertex which has threshold $r$ having observed at least $r$ other vertices adopting a new behaviour it adopts it too. The distribution of the thresholds among the members of the population is assumed to have cdf $F$.This process is assumed to take place on the complete graph on $n$ vertices and 
$F(x)$ denotes the fraction of vertices that have threshold at most $xn$. If after round $t$ there is an $r(t)$ fraction of individuals that are active and $F(r(t)) > r(t)$, then $r(t+1)= F(r(t))$. This process reaches an equilibrium which is a fixed point of $F$. 
Processes of similar flavour were also described in~\cite{ar:ClifSud1973,ar:EllFud1993,ar:VenkSanj1998}.

In physics, threshold models where all vertices have the same activation threshold have been used in order to 
explain magnetisation in materials whose molecular structure is described by a lattice. Such is  the class of 
bootstrap processes, which were introduced by Challupa, Leath and Reich~\cite{ar:CLR79} in that 
context, but have generated a long line of research within combinatorics and probability theory. A key difference between these two models and the process that is our main focus is that the former are \emph{monotone}: once a vertex is activated it remains active. However, in the class of processes we will consider a vertex may switch between states an arbitrary number of times.

Another threshold model was developed in the 1940s by McCulloch and Pitts~\cite{ar:McCullochPitts1943} to describe neuronal activity. In this model, a neuron receives positive and negative inputs and it fires, if their sum is positive. Neurons are interconnected forming a large hierarchical network. 

From a purely mathematical perspective, it would make sense to study the behaviour of such a process in a general graph, which may be thought of as a representation of a population.  A simpler nevertheless non-trivial version of the above would be to consider a population represented by a graph and there are two opinions/states available; during each round each vertex adopts an opinion if at least half of the neighbours have adopted it. This modification gives rise to \emph{majority dynamics} with $k=2$ states.  

When two states are available one can adopt a deterministic resolution rule, whereby a vertex does not change 
state if there is a tie among the numbers of neighbours having each one of the two states. In this case, a result of Goles and Olivos~\cite{ar:GolOliv80} implies that if the underlying graph is finite, then the system will eventually become periodic with period at most 2. In other words, there is a limiting behaviour as the number of rounds grows and one may further ask how does this limit look like.

A natural question is about the effect the geometry of the underlying graph has to the 
limiting behaviour.   
Mossel et al.~\cite{ar:MosselNeemanTamuz} considered majority dynamics with two states, say 0 and 1, where each vertex is initially assigned state 0 with probability $1/2 + \delta$, independently of any other vertex, for some 
fixed $\delta >0$.  
They consider a classification of the vertices into \emph{social types}
where two vertices have the same social type if there is an automorphism of the graph that maps one vertex to the other. In other words, vertices of the same social type are indistinguishable within $G$.  Let $m(G)$ be the size of the smallest class of the social types. Their main result shows that the probability that the dynamics arrives at unanimity of one of the two states converges to 1 polynomially fast in $1/ m(G)$ with exponent proportional to $\delta / \log (1/\delta)$.
(The social types are not directly related to the states as the initial assignment of states is independent of the social type.)
On the other hand, they show the existence of a sequence of graphs $(G_n)_{n \in\mathbb{N}}$ on $n$ vertices for which this probability stays bounded away from 1 as $n\to \infty$. In this example, $m(G_n)$ is small and does not tend to infinity. 

The other parameter they considered is the expansion of the underlying graph. In particular, Mossel et al.~\cite{ar:MosselNeemanTamuz} considered a sequence $(G_n)_{n\in \mathbb{N}}$ of $d$-regular graphs which are $\lambda$-expanders (the absolute value of the second largest eigenvalue of their adjacency matrix is no larger than $\lambda$) 
with $\lambda/d \leq 3/16$. 
In majority dynamics with $k\geq 2$ states $\{1,\ldots,k\}$ starting from a random initial configuration where the probability of state $1$ exceeds the probability of state $i\not =1$ by at least $c\log k / \sqrt{d}$, they show that state $1$ eventually dominates and unanimity is achieved with high probability as $n\to \infty$.

Benjamini et al.~\cite{ar:BenjChanmDonnelTamuz2016} considered this question in the context of the binomial 
random graph $G(n,p)$, where each pair of vertices from the set $V_n = \{1,\ldots, n\}$ appears as an edge with probability $p =p(n)$ independently of any other pair and the initial state is selected uniformly at random. 
They conjectured that if $np \to \infty$ as $n\to \infty$, then this dynamics with two states starting from a uniformly random configuration will arrive at almost unanimity with one of the two states occupying at least  
$(1-\eps)n$ of the vertices with high probability as $n\to \infty$, for any $\eps >0$. 
They proved that unanimity is attained for $np\gg n^{1/2}$ with probability that is asymptotically 
bounded away from 0. Later Fountoulakis et al.~\cite{ar:FMK_RSA_2021} showed that this is the case with 
probability tending to $1$ as $n\to \infty$. Moreover, they showed that unanimity occurs in at most 4 rounds. 
This result was refined by Berkowitz and Devlin~\cite{ar:BerkDev_SPA}. They showed that for any given initial configuration, the number of vertices in one of the two states after the execution of one round satisfies a central limit theorem (over the random choice of the graph). Using this result, they obtained more precise results for the number of rounds until unanimity is achieved, when the initial configuration is uniformly random, provided that $np\gg n^{1/2}$.  Tran and Vu~\cite{ar:TranVu2020} also gave another proof of the main theorem of~\cite{ar:FMK_RSA_2021}. 
Furthermore, Chakraborti et al.~\cite{ar:CKLT2021} reduced the value of $p$ for which unanimity is reached by majority dynamics under a uniformly random initial configuration to $p\gg n^{-3/5} \log n$. 
More recently, Tran and Kim~\cite{ar:TranKim25} and independently Jaffe~\cite{ar:Jaffe25} reduced this lower bound to $p \gg n^{-2/3}$.

Tran and Vu~\cite{ar:TranVu2020} also explored non-random initial configurations and, in particular, configurations 
where one of the states has a very small majority over the other. They showed that when the number of vertices in the most popular state is at least 6 more than the number of vertices in the other state, majority dynamics on $G(n,1/2)$ reaches unanimity with the dominating state with probability at least $0.9$. They extended this to $G(n,p)$ for any fixed $p \in (0,1)$ and majority $c_n$ (possibly depending on $n$).  
This result was strengthened by Sah and Sawhney~\cite{ar:SahSawney2020}, who gave a precise estimate on 
the probability that the initial majority of one of the two states becomes dominance of that state, for values of $p$ such that $(\log n)^{-1/16}\leq p \leq 1- (\log n)^{-1/16}$. 

On the other hand, initial configurations with relatively high initial majority have been proved to be reinforced and 
lead to unanimity for much smaller values of $p$. For example, Tran and Vu~\cite{ar:TranVu2020} showed that if 
$np\geq (2+o(1)) \log n$ an initial majority of order $1/p$ is re-enforced to unanimity with high probability as $n\to\infty$. Zehmakan~\cite{ar:Zehmakan2018} had shown this for $np> (1+o(1))\log n$ (that is, just above the connectivity threshold) and an initial majority that is asymptotically larger than $(n/p)^{1/2}$. In the same work, 
Zehmakan refined the results of Mossel et al.~\cite{ar:MosselNeemanTamuz} for $d$-regular expander graphs.  
Earlier versions of this result were shown by G\"artner and Zehmakan in~\cite{ar:GarZeh2018} and in~\cite{ar:GarZeh2017} (where the underlying graph is a grid). 

The asynchronous version of majority dynamics with two states in which vertices update their states one-at-a-time 
was studied by Mohan and Pra{\l}at in~\cite{ar:MoPr2023}. More precisely, at each round a vertex is selected 
uniformly at random and updates their state, where in the case of a tie the vertex does not change state.  
The initial state is not selected uniformly at random but there is a $\delta$-bias towards one of the two states, for some fixed $\delta >0$. 
They show that when the underlying graph is a random graph distributed as $G(n,p)$ with $np \gg \log n$ as $n\to \infty$ but $p=o(1)$, this state dominates. However, this is not the case when $p= \Omega (1)$. They show that the other state may dominate with probability that is asymptotically bounded away from $0$. 

\subsection{Results}
The aim of this contribution is to consider the synchronous process for $k$ states with $k>2$ and study the limiting behaviour of $S_t$ as $t\to\infty$ when the underlying graph is $G(n,p)$, the \emph{binomial random graph} on $n$ vertices.  As it is customary in the theory of random graphs, we will be interested in asymptotics as $n \to \infty$. 
Let $\boldsymbol{\lambda} := (\lambda_1,\ldots, \lambda_k)$ be such that $\sum_{i=1}^k \lambda_i =1$ and $\lambda_i >0$, for $i=1,\ldots, k$; we write that $\boldsymbol{\lambda} \in \triangle_1$.
The initial configuration $S_0$ will be selected randomly, where each $v\in V_n$ will have $S_0(v)=i \in \Scal_k$ with probability $\lambda_i$, independently of all other vertices and the edges of $G(n,p)$. 
We write $S_0 \sim \boldsymbol{\lambda}$ to denote this random choice. 

Our main result is that for any choice of $\boldsymbol{\lambda}\in \triangle_1$, if $p$ is not too rapidly decaying as a function of $n$, then the dynamics will reach unanimity in a bounded number of rounds with high probability.  
\begin{theorem} \label{thm:main} 
Let $\eps >0$ and $\boldsymbol{\lambda} \in \triangle_1$. 
Consider majority dynamics on $G(n,p)$ with $np\gg n^{2/3}$ but $\limsup_{n\to \infty} p < 1$. For any $n$ sufficiently large, 
with probability at least $1-\eps$ the following holds:
if $S_0 \sim \boldsymbol{\lambda}$, 
then $S_3 (v)=s$, for all $v\in V_n$, for some $s\in \Scal_k$. 
\end{theorem}
The lower bound on $p$ appears also in the context of the analysis of the label propagation algorithm on the 
binomial random graph by Kiwi et al.~\cite{ar:KLMP2023}. This is an algorithm whose purpose is community 
detection in a network. Initially, each vertex selects a label from $[0,1]$ uniformly at random independently of any 
other vertex. (So with probability 1, all vertices begin with different labels, but this may not be the case in subsequent rounds.) Thereafter, 
the labels are updated 
synchronously by adopting the label that appears most frequently in the neighbourhood of each vertex, but 
resolving ties towards the smallest label, in the first round, and randomly in subsequent rounds. Among other results, the authors show that if $np\gg n^{2/3}$, 
then a.a.s. the smallest label dominates the entire random graph.  

\subsection{Outline of the proof and the structure of this work}

The proof of Theorem~\ref{thm:main} relies on a careful analysis of the first round.  
\begin{enumerate}
\item We begin with showing that with high probability during the first round all states but those which are in $\Mcal_0:=\argmax \{\lambda_i\}_{i\in \Scal_k}$ are eliminated - this is Claim~\ref{clm:only_M_0}. 
We do this using a first moment argument, showing that with high probability for any two states $i, j \in \Scal_k$ with $\lambda_i < \lambda_j$,  all vertices  have more neighbours at state $j$ in $S_0$ than in state $i$.

\item  Thereafter, for any vertex $v\in V_n$ and any $i \in \Mcal_0$, we estimate the 
probability that $S_1 (v) = i$. To this end we show first that with high probability a small number of 
vertices face a tie during the execution of the first round (Claim~\ref{clm:residual_set}).  
So most vertices adopt their state $i \in \Mcal_0$ in round 1 through a clear majority within the corresponding subset of vertices that have been set at state $i$ initially. 
For such a vertex $v$ we say that $S_1(v)=i$ in \emph{a strong sense}.

\item  We estimate the probability that  $S_1(v)=i$ in a strong sense in Section~\ref{sec:exp_round_one}. In particular, we estimate this probability conditional on the size of the neighbourhood of $v$ among the vertices who are initially in a state in 
$\Mcal_0$. The main tool that we deploy here is a local limit theorem (Lemma~\ref{lem:llt}), through which we obtain an asymptotic estimate on the probability that $v$ has a certain number of neighbours in each subset of vertices that are initially at some state in $\Mcal_0$. 
The other tool that we use is an anti-concentration result (Lemma~\ref{lem:anti_conc}) for the initial configuration, which we prove in Section~\ref{sec:anticonc}. 
The latter states that with high probability any two parts of the initial configuration differ in size by $\Omega (n^{1/2})$. 
We combine these tools in Lemma~\ref{lem:conclusion} to deduce a lower bound on the difference between the probability that 
$v$ takes the state corresponding to the $i$th largest part after the first round in a strong sense and the probability that it takes the state 
which corresponds to the $i+1$st largest part. 
\item Thereafter, a second moment argument in Section~\ref{sec:second_moment} shows that the number 
of vertices that adopt state $i\in \Mcal_0$ is concentrated around its expected value. So after the execution of  the first round, the vertex set is partitioned into as many sets as $|\Mcal_0|$ together with a smaller set that consists of the vertices that face a tie during the first round. 

\item Ordering those parts that correspond to states in $\Mcal_0$ according to their size, we conclude that consecutive parts differ in size by $\Omega (np^{1/2})$. 
The problem that arises at this stage of the proof is that we have exposed all edges of the random graph.
To bypass this issue, Lemma~\ref{lem:nbds_conc} in Section~\ref{sec:after_round_1} uses a union bound in order to show that most vertices will have many neighbours in the largest among those parts and the corresponding state will dominate after the execution of the second round. It is for the union bound to work that we need the assumption that $np\gg n^{2/3}$. This is a general lemma which implies that with high probability for any partition which is similar to the partition we end up with after the first round most vertices will have most of their neighbours inside the largest part. 
\item We have proved that all vertices in $V_n$, except for a subset of size $o(n)$, will have adopted that state. We finish the proof showing that during the third round all vertices will have switched to that state. 
\end{enumerate}

\subsection{Some notation and concentration bounds} 

We denote by $\mathbb{N}_0$ the set of non-negative integers, that is $\mathbb{N}_0= 
\{0\} \cup \mathbb{N}$.  

For a graph $G=(V,E)$ a vertex $v \in V$ and a subset $S \subseteq V$, we let $N_S(v)$ be the set of neighbours of $v$ in $S$. We also set $N_G(v) = N_V (v)$, that is, the set of all neighbours of $v$ in the graph $G$.  
The binomial random graph $G(n,p)$ will have $V_n$ as its vertex set. 

For $i \in \Scal_k$, let  $S_0^{-1}(i) = \{v\in V_n \ : \ S_0(v)=i \}$, and $n_i = |S_0^{-1}(i)|$. 
For two distinct vertices $v,v'\in V_n$, we set $n_{i,v} = |S_0^{-1}(i)\setminus \{v\}|$ and 
$n_{i,v,v'} = |S_0^{-1}(i) \setminus \{v,v'\}|$. 

Furthermore, for any vertex $v\in V_n$, we denote by $n_i(v)$ the number of neighbours of $v$ in 
$S_0^{-1}(i)$, that is, $n_i(v) := |N_{S_0^{-1}(i)}(v)|$. For an arbitrary subset $S \subseteq V$, we let 
$d_S(v) = |N_S (v)|$; we will use the latter primarily in Section~\ref{sec:after_round_1}. 

As  we have already pointed out in the outline of the proof of our main theorem, it is the classes 
$\{S_0^{-1}(i)\}_{i\in \Mcal_0}$ that play a significant role, as every other state will disappear after the first 
round.  We will denote by $S^*$ the set $\cup_{i\in \Mcal_0} S_0^{-1}(i)$ and let $n^* : = |S^*|$. Also, 
we set $S_* = V_n \setminus S^*$. 
For two distinct vertices $v,v'\in V_n$, we set $n_{v}^* = |S^* \setminus \{v\}|$ and 
$n^*_{v,v'} = |S^* \setminus \{v,v'\}|$. 

For a vertex $v \in V_n$, we set $n^*(v) : = |N_{S^*}(v)|$. 
Similarly, we set $n^*_{v}(v') = |N_{S^*\setminus \{v\}} (v')|$ and 
$n^*_{v,v'}(v) = |N_{S^*\setminus \{v,v'\}} (v)|$. In Appendix~\ref{app:second_moment}, we will use the notation $n_{j,v'} (v) = |N_{S_0^{-1}(j)\setminus \{v'\}} (v)|$.

For $n\in \mathbb{N}$ and $p\in [0,1]$, we denote by $\bin (n,p)$ the binomial distribution with parameters $n$ and $p$. For a vector $\q = (q_1,\ldots, q_{k-1})$ such that $(q_1,\ldots, q_{k-1}, 1-\sum_{j<k} q_j)\in \triangle_1$, 
we denote by $\mult (n;\q)$ the multinomial distribution on a collection of i.i.d. random variables that has size $n$, where each member of it takes value $i\in \{1,\ldots, k\}$ with probability $q_i$, if $i < k$, and $1-\sum_{j<k} q_j$, 
for $i=k$. 

We will be using the standard Chernoff bound for the concentration of binomially distributed 
random variables. The inequality we will use follows from Theorem 2.1 in~\cite{bk:JLR2000}.
If $X$ is a random variable such that  $X\sim \bin (n,q)$, then for any $\delta \in (0,1)$ we have 
\begin{equation} \label{eq:Chernoff}
	\prb{|X - nq| \ge \delta nq} \le 2 e^{-\delta^2 nq / 3}.
\end{equation}
A Chernoff-type bound also holds for the hypergeometric distribution $\hyp (n,m;s)$. Suppose that on a groundset $X_N$ of size $N$ a subset $S$ of size $s$ is selected uniformly at random among all subsets of $X_N$ of this size. Let $X_M \subset X_N$ be a subset that has size $M$. 
Then $\hyp (N,M;s)$ is the distribution of the random variable that is the size of $S \cap X_M$. Note that 
$\E{|S\cap X_M|}= s (M/N)$. For the concentration of $|S\cap X_M|$ around its expected value we will use a 
Chernoff-type inequality by Serfling~\cite{ar:Serfling1974} (in fact, a special case of it): 
\begin{equation} \label{eq:Chernoff_hyper}
	\prb{|S\cap X_M|- s (M/N)| \ge \lambda \sqrt{s}} \le 2 e^{-2\lambda^2 / (1-f_s)},
\end{equation}
where $f_s= (s-1)/N$. 

We will also use the local limit theorem for the binomial distribution - see for example the monograph by Petrov~\cite{bk:Petrov1975} and a generalisation of it in~\cite{ar:FMK_RSA_2021}.
\begin{theorem} \label{thm:LLT}
Let $X\sim \bin (n,q)$ for some $q \in (0,1)$. 
Then with $\sigma^2 = nq(1-q)$ we have 
$$\sup_{k\in \mathbb{Z}} \left| \prb{X=k} - \frac{1}{\sqrt{2\pi \sigma^2}} e^{-\frac{(k - np)^2}{2\sigma^2}} \right| = O\left(\frac{1}{\sigma^2}\right). $$
\end{theorem}

Finally, we say that a sequence of events $E_{n}$ defined on a sequence of probability spaces with probability measure 
$\mathbb{P}_n$ occurs \emph{asymptotically almost surely} (or \emph{a.a.s.}) if 
$\mathbb{P}_{n}(E_{n}) \to 1$ as $n\to \infty$.

\section{Some observations about the initial configuration} 

\subsection{On the concentration of state classes and the vertex degrees therein} 
The definition of the initial state implies that the vector $(n_1, \ldots, n_k)$ follows the multinomial 
distribution $\mult (n;(\lambda_1,\ldots, \lambda_{k-1}))$.
In particular, $n_i \sim \bin (n,\lambda_i)$ for any $i \in \Scal_k$.

\begin{claim} \label{clm:good_init_state} With probability at least $1-o(n^{-1})$, for all $i \in \Scal_k$ we have $|n_i - n \lambda_i |\leq \log n \sqrt{n}$.
\end{claim}
\begin{proof}
The Chernoff bound~\eqref{eq:Chernoff} 
implies that $\prb{ |n_i - n \lambda_i |> \log n \sqrt{n}} \leq \exp\left(-\Omega(\log^2 n)\right)$.
Now, the union bound over all $i \in \Scal_k$ (as $k$ does not depend on $n$)
implies that a.a.s. for all $i \in \Scal_k$ we have $|n_i - n \lambda_i |\leq \log n \sqrt{n}$. 
\end{proof}
We denote this event $\Ev_0$ and write $S_0 \in \Ev_0$ to say that $\Ev_0$ is realised. (The choice of the $\log$ function is somewhat arbitrary as $k$ is fixed. Here and below, we will not try to make optimal choices, unless we need to do so.)
Let $\lambda_* = \sum_{i\not \in \Mcal_0} \lambda_i$ and, with $k_0 = |\Mcal_0|$, $\lambda^* = \sum_{i \in \Mcal_0} \lambda_i = k_0 \lambda_{\max}$. 
Recall that we let $n^*=|S^*|$.
Claim~\ref{clm:good_init_state} implies that 
with probability $1-o(n^{-1})$,  $\left| n^* - n \lambda^* \right| < k_0 n^{1/2} \log n$. 
We denote this event by $\Pcal_0$. 
(Note that if $\Ev_0$ is realised, then $\Pcal_0$ is realised as well.)

Another important random variable is $n_i(v)$ for $i\in \Scal_k$ and any $v\in V_n$. 
Conditional on an initial configuration $\{ S_0^{-1}(1),\ldots, S_0^{-1}(k)\}$, we have $n_i(v)\sim \bin (n_i - 1_{S_0(v)=i},p)$, where $ 1_{S_0(v)=i}=1$, if $S_0(v)=i$ and 0 otherwise.
\begin{claim} \label{clm:deg_conc_vc}
Conditional on an initial configuration $S_0 \in\Ev_0$,
we have that a.a.s. all vertices $v\in V_n$ have $|n_i(v) - n_i p|\leq \log n \sqrt{np}$ for all $i \in \Scal_k$.
\end{claim}
\begin{proof}
This 
follows from the Chernoff bound and Markov's inequality. Indeed, conditional on a configuration $\{ S_0^{-1}(1),\ldots, S_0^{-1}(k)\}$ that satisfies $| n_i - n \lambda_i |\leq \log n \sqrt{n}$ for all $i \in \Scal_k$ (that is, which  
satisfies $\Ev_0$), the Chernoff bound yields that for a vertex $v\in V_n$
$\prb{| n_i(v)  - n_i p| >  \log n \sqrt{np}} = \exp \left(-\Omega (\log^2 n)\right) = o(n^{-1})$. 
The union bound (over all vertices and states $i\in \Scal_k$) and Markov's inequality imply that a.a.s. all vertices $v\in V_n$ satisfy $|n_i(v) - n_i p|\leq \log n \sqrt{np}$.
\end{proof}

Recall also that $N_{S^*}(v)$ denotes the neighbourhood
of $v$ in $S^*$ and $n_v^*= |S^*\setminus \{v\}|$. 
Then the random variable $n^* (v) = |N_{S^*} (v)|$ has $n^* (v) \sim \bin (n_v^*,p)$. 
We will consider the event 
$\Ncal^*_{v}:= \{|n^* (v)- n_v^* p|\leq k_0 \log n \cdot (np)^{1/2}\}$. 

%Consider a vertex $v \in V_{n}$ and let $S_v^* = S^* \setminus \{v\}$ having 
%size $n_v^*= n^*-\ind{v \in S^*}$, where $\ind{v \in S^*}=1$ if $v \in S^*$, but it is equal to $0$ otherwise. 

With a slight abuse of notation, for an event $\mathcal{E}$ and $\ind{\mathcal{E}}$ denoting its indicator function, and random variables $X_1,\ldots , X_m$ we write,
\[
\mathbb{P}\left[\mathcal{E} | X_1,\ldots, X_m \right] = \E{\ind{\mathcal{E}} | X_1,\ldots, X_m}.
\] 
We saw that on the event $\Ev_0$ we have $|n^* - n \lambda^*| <  k_0 n^{1/2} \log n$. 
So Claim~\ref{clm:deg_conc_vc} and its proof yield that on the event $\Ev_0$ 
\begin{equation} \label{eq:Nv_prob}
\prb{\Ncal^*_{v} | n^*} = 1-o(n^{-1}). 
\end{equation}

We next discuss a bound on the (conditional) probability that a vertex $v$ has the same number of neighbours in two 
parts $S_0^{-1}(i)$ and $S_0^{-1}(j)$, for $i\not = j$. This will be useful as any vertex which ends up choosing its state during 
round 1 through the resolution of a tie should have the same number of neighbours in at least two parts. 
Let $T_1$ denote the set of vertices that face a tie during round 1.

\begin{claim} \label{clm:residual_set}
The set $T_1 \subset V_n$  a.a.s. has
$$ |T_1| \leq \log^2 n \sqrt{\frac{n}{p}}.$$
\end{claim}
\noindent
We call this event $\Ev_{T_1}$.
\begin{proof}
Note that conditional on a realisation of $\{S_0^{-1}(1),\ldots, S_0^{-1}(k)\}$ the random variables $n_\ell (v)\sim \bin (n_\ell - 1_{S_0(v)=\ell},p)$, for $\ell \in \{i,j\}$, are independent. 
In particular, we condition on  a realisation of $\{S_0^{-1}(1),\ldots, S_0^{-1}(k)\}$ that satisfies $\Ev_0$; by Claim~\ref{clm:good_init_state} the probability that this does not occur is $o(n^{-1})$. 
To bound the (conditional) probability that $n_i (v) = n_j (v)$ we will make use of the local limit theorem (Theorem~\ref{thm:LLT}). 
Firstly, the proof of Claim~\ref{clm:deg_conc_vc} implies that with probability $1-o(n^{-1})$ we have 
$|n_\ell (v) - (n_\ell - 1_{S_0(v)=\ell})p|\leq \log n \sqrt{np}$ for $\ell \in \{i,j\}$. 
 Theorem~\ref{thm:LLT} implies that for $s_i \in \mathbb{N}$ such that $|s_i -  (n_i - 1_{S_0(v)=i})p|\leq \log n \sqrt{np}$ 
we have that 
\begin{eqnarray*}
\prb{n_i(v)= n_j(v) =s_i \mid S_0} =
\prb{n_i (v)=s_i \mid S_0} 
\prb{n_j (v)=s_i \mid S_0}= O\left(\frac{1}{np}\right).
\end{eqnarray*}
Moreover, this upper bound is uniform over these choices of $s_i$. As there are $O(\log n \sqrt{np})$ choices for the value 
of $s_i$ we conclude that 
\begin{equation} \label{eq:prb_eq_nbs}\prb{n_i(v)= n_j(v)} = O\left(\frac{\log n}{(np)^{1/2}}\right)+o(n^{-1}) = O\left(\frac{\log n}{(np)^{1/2}}\right).
\end{equation}
It follows that 
\begin{eqnarray*} \E{|T_1|} &\leq \sum_{v \in V_n} \prb{\exists i\not =j \ n_i (v)=n_j(v)}\leq 
 \sum_{v \in V_n} \sum_{i\not = j \in \Scal_k} \prb{n_i (v)=n_j (v)}  \\
&\stackrel{\eqref{eq:prb_eq_nbs}}{=} O\left( k^2 \log n\sqrt{\frac{n}{p}}\right) = O\left( \log n\sqrt{\frac{n}{p}}\right).
 \end{eqnarray*}
Markov's inequality then implies the claim.
\end{proof}

Recall that we set $\Mcal_0 = \argmax \{\lambda_i\}_{i \in \Scal_k}$ and now let $\lambda_{\max} = \max \Scal_k$; thus, any $i \in \Mcal_0$ has $\lambda_i = \lambda_{\max}$. 
We will show that all states outside $\Mcal_0$ are eliminated in the first round and only the states in 
$\Mcal_0$ survive a.a.s. This means that, with high probability, what determines $S_1(v)$, for any $v \in V_n$ are the random variables $\{n_i(v)\}_{i \in \Mcal_0}$.
\begin{claim} \label{clm:only_M_0}
 A.a.s. all $v\in V_n$ have  $S_1(v) \in \Mcal_0$.
\end{claim}
\begin{proof}
Firstly, let us observe that Claim~\ref{clm:good_init_state} implies that there exists $\eps>0$ such that a.a.s. for any $j \not \in \Mcal_0$ and any $i\in \Mcal_0$, we have  $\frac{1}{n} \left( n_i - n_j \right) >\eps$. 

Let us condition on a realisation of $\{S_0^{-1}(1),\ldots, S_0^{-1}(k)\}$ such that the above event is realised as well. 
A simple application of the Chernoff bound (cf. the proof of Claim~\ref{clm:deg_conc_vc}) together with the union bound imply that a.a.s. 
all vertices $v\in V_n$ have 
$n_j (v) < n_i(v )$ for any $j \not \in \Mcal_0$ and any $i \in \Mcal_0$. Thus, the claim holds. 
\end{proof}

%We will exploit this observation using a two-round exposure of the multinomial distribution $\mult (n; \la)$. 
%
%Recall that we set $S_* = \cup_{i \not \in \Mcal_0} S_0^{-1} (i)$, $S^* = \cup_{i \in \Mcal_0} S_0^{-1}(i)$ and $n^* = |S^*|$. 
%Note that $|S^*| \sim \bin (n,\lambda^*)$. 

Another useful bound will be on the absolute difference $|n_i (v)- n^*(v) (n_i/n^*)|$, having set $n^*(v)=|N_{S^*}(v)|$, conditional on a realisation of $S_0 \in \Ev_0$. 
Conditioning on $S_0\in \Ev_0$ and 
the value of $n^*(v)$, the neighbourhood of $v$ in  $S^*$ can be any subset of $S^*$ of size 
$n^*(v)$ with the same probability. Therefore, for any $i\in \Mcal_0$ we have 
$n_i(v) \sim \hyp (n^*,n_i;n^*(v))$. 
\begin{claim} \label{clm:hypergeom_no_large_cs}
Let $v\in V_n$. 
On the events $\Ev_0$ and $\Ncal^*_{v}$ we have 
\begin{equation*}
\prb{|n_i (v)- n^*(v) (n_i/n^*)|\geq \log n\sqrt{n^*(v)}  \mid S_0, n(v)^*} =\exp \left(-\Omega (\log^2 n) \right). 
\end{equation*}
for all $i \in \Mcal_0$. 
\end{claim}
\begin{proof}
For a given vertex $v\in V_n$, we condition on the values of $n_i, n_v^*$, so that $\Ev_0$ (and therefore $\Pcal_0$) are satisfied, and $n^*(v)$ being such that $\Ncal_v^*$ is satisfied. But $|n_v^* - n^*|\leq 1$. 
So on $\Ncal_v^*$ we have $n^*(v)/n_v^* \sim n^*(v)/n^* \sim p$. 
We apply the Chernoff-type bound for the hypergeometric distribution~\eqref{eq:Chernoff_hyper} with 
$\lambda = \log n$, $s=n^*(v)$,  $M= n_i$ and $N=n_v^*$. 
So $f_s = s/N =n^*(v)/n_v^* \sim p$. But as $\limsup_{n\to \infty} p <1$, we have $\lim\inf_{n\to \infty} (1- f_s) >0$. 
With these parameters,~\eqref{eq:Chernoff_hyper} yields
$$\prb{|n_i(v)- n^*(v) (n_i/n^*)|\geq \log n\sqrt{n^*(v)} \mid n_i, n^*,n^*(v)} = \exp \left(-\Omega (\log^2 n) \right).$$
%As this probability is $o(n^{-1})$, the union bound over all vertices $v\in V_n$ yields the statement of the claim. 
\end{proof}
This lemma suggests that conditional on the values of $n_i, n^*_v$ and $n^*(v)$ on the events $\Ev_0$ and 
$\Ncal^*_{v}$, the random variable $n_i(v)$ is concentrated around its expected value in an interval 
of size $2 (n^*(v))^{1/2}\log n$.  In the next section, we will give the detailed asymptotic distribution of the vector $(n_i(v))_{i\in \Mcal_0}$ within this range.

\subsection{A local limit theorem}  \label{sec:LLT}
Without loss of generality, suppose that $\Mcal_0 = \{1,\ldots, k_0\}$. 
If we condition on a realisation of $S_*, S^*$, then each vertex $v\in S^*$ will take state $i \in \Mcal_0$ with 
probability $1/k_0$ independently of any other vertex in $S^*$ as well as of the random graph on $V_n$. 
Thus, it is natural to express the sizes of the sets $S_0^{-1}(i)$, for $i \in \Mcal_0$, in terms of their deviation from $n^*/k_0$. 

Let $c_1,\ldots, c_{k_0} \in \mathbb{R}$ be such that $n_i = n^*/k_0 + c_i \sqrt{n^*}$; thus $\sum_{i=1}^{k_0}c_i=0$.
%For each such $i\in \Mcal_0$ we set $S_v^*_i = S_v^* \cap S_0^{-1} (i)$, that is, the set of vertices in 
%$S^* \setminus \{v\}$, which have state $i$ initially.  
%These sets form a partition on $S^*\setminus \{v\}$, where each vertex $u \in S^*\setminus\{v\}$ 
%belongs to $\hat{S}^*_i$ precisely when $S_0(u) = i$.  
As $n_{i,v}=|S_0^{-1} (i) \setminus \{v\}|$, it follows that $n_{i,v} = n^*/k_0 + c_i\sqrt{n^*} +O(1)$. 
Note that $n_i-1\leq n_{i,v} \leq n_i$. 
Furthermore, we write $n_{i,v} = n_v^* p_i$. 
Since $n_v^* =n^*-1_{v\in S^*}$, it follows that $p_i = 1/k_0 + c_i/\sqrt{n^*} +O(n^{-1})$.

Let us condition on $n^* (v)$ taking a particular value such that $\Ncal^*_{v}$ occurs. 
For $i=1,\ldots, k_0-1$ let us write $m_i = \lceil n^* (v) p_i\rceil$ and 
$m_{k_0} = n^*(v) - \sum_{i=1}^{k_0-1}m_i$. The quantities $m_i$ are essentially the expected number of 
neighbours that $v$ has inside $S_0^{-1}(i)$, if we condition on $n^*(v)$ and on a realisation of $S_0^{-1}(1),\ldots, S_0^{-1}(k_0)$. Lemma~\ref{lem:llt} below provides a local limit theorem for $n_1(v),\ldots, n_{k_0}(v)$ when these are 
centralised around $m_1,\ldots, m_{k_0}$.
%Also, note that $n_0(v;i) = |N_{S^*}(v)  \cap S_0^{-1} (i)|$, for $i=1,\ldots, k_0$.  

We will give an estimate on the probability that 
$(n_1(v),\ldots, n_{k_0}(v)) = (s_1,\ldots,s_{k_0})$ conditional on the values of $n_v^*, n^*(v)$ and on  
$\cv^* := (c_1,\ldots, c_{k_0})^T\in \mathbb{R}^{k_0}$. (Note that the random vector 
$\cv^*$ determines the deviation of $n_i$ from $n^*/k_0$, for $i=1,\ldots, k_0$.)
This probability is: 
\begin{equation} \label{eq:probability}
\prb{(n_1(v),\ldots, n_{k_0} (v)) = (s_1,\ldots,s_{k_0}) \mid n_v^*, n^*(v), \cv^*} = \frac{\binom{n_{1,v}}{s_1} \cdots \cdot \binom{n_{k_0,v}}{s_{k_0}}}{\binom{n_v^*}{n^*(v)}}.
\end{equation}
We will show that this probability is approximated by the
probability density function of a multivariate gaussian
random vector, provided that $s_i$ is close to $m_i$ for all 
$i=1,\ldots, k_0-1$. 
%We remark that such a probability is dependent on the value of $n_{L_n}(v)$ and the initial global strategies $\left\{S_0{(1)}, \ldots,S_0{(k_n)}\right\}$. 
Let $\un_{\ell}$ denote the $\ell \times \ell$ unit matrix and $\mathbf{1}_{\ell \times \ell}$ the all-ones $\ell\times \ell$ matrix. 
\begin{lemma} \label{lem:llt}
Let $p=p(n) \in [0,1]$ be such that $\limsup_{n\to \infty} p(n) <1$ and $n^{1/2}p \geq \log^3{n}$. 
Consider a realisation of $S_0 \in \Ev_0$, and for $v \in V_n$
let $n^*(v)$ be such that $\Ncal^*_{v}$ occurs. 
Suppose that $s_i = m_i+\delta_i$ with $|\delta_i| \leq \log n \cdot (n^*(v))^{1/2}$, for $i=1,\ldots, k_0-1$. Then uniformly over these values of $(\delta_1,\ldots,\delta_{k_0-1})$ and $\cv^*$ such that $\| \cv^*\|_\infty \leq \log n$,
with $\x = (x_1,\ldots, x_{k_0-1})^T := (\delta_1,\ldots , \delta_{k_0-1})^T / n^*(v)^{1/2}$ and 
\[\Sigma= \frac{k_0}{1-n^*(v)/ n_v^*} \left(\un_{k_0-1} + \mathbf{1}_{(k_0-1)\times (k_0-1)}\right),
%\left[ \begin{array}{cccc} 
%2 &1 & \cdots & 1 \\
%1 &2 &\cdots &1 \\
%&\vdots & \cdots & \vdots \\
%1& \cdots &2&1 \\
%1 & \cdots &1& 2 
%\end{array}\right],
\] we have 
\begin{eqnarray*}\lefteqn{ \prb{(n_1(v),\ldots, n_{k_0}(v) ) = (s_1,\ldots,s_{k_0}) \mid n_v^*, n^*(v), \cv^*} =}
\\
&& \hspace{3.5cm} \frac{1}{n^*(v)^{\frac{k_0-1}{2}}} \cdot \phi(\x)\left(1+o(p^{1/2}) \right),
\end{eqnarray*}
where $\phi (\x) = \left(\frac{|\Sigma|}{2\pi}\right)^{\frac{k_0-1}{2}} \exp \left(- \frac12 \x^T \Sigma \x \right)$ and 
$|\Sigma| = \mathrm{det} (\Sigma)$.
\end{lemma}
\begin{proof}
For notational convenience we will write $n(v)$ for $n^*(v)$, $k$ for $k_0$ and $n$ for 
$n_v^*$ and $n_i$ for $n_{i,v}$. This should cause no confusion in regards to the asymptotics we derive,
as $n^*(v)$ is a.a.s. of the same order of magnitude as $n(v)$, and $n_{i,v}$ differs by at most 1 from $n_i$, respectively.
We will give asymptotic estimates for each one of the binomial coefficients that appear in~\eqref{eq:probability}. 
Before we do this, let us give some 
relations, which we will use in our 
estimates later. 
Firstly, let us observe that since $np \gg n^{2/3}$
\begin{equation} \label{eq:est1}
    \frac{\log^2 n}{np} \ll \frac{\log^{3}{n}}{\sqrt{np}}\leq p^{1/2}.
\end{equation}
%This is equivalent to $1/(np)^2 = o(p)$ which in turn is equivalent to $1/n^2 = o(p^3)$. But this holds, since $1/n^{1/2} = o (p)$.
Therefore, 
\begin{equation}
     \label{eq:est3}
    \frac{\log^2 n}{np} = o \left(p^{1/2}\right).
\end{equation}
This is equivalent to $\log^2 n / n = o (p^{3/2})$. But
$$\frac{\log^{4/3} n}{n^{2/3}} \ll \frac{1}{n^{1/2}}\ll p.$$

%Similarly, 
%\begin{equation}
%     \label{eq:est2}
%    \frac{\omega(n)^3}{(np)^{1/2}} = o \left(p^{1/2}\right).
%\end{equation}
%This is the case, as 
%\begin{equation*}
%    \frac{\omega(n)^3}{(np)^{1/2}} 
%    = o\left(\frac{\omega (n)^4}{ (np)^{1/2}}\right)
%\end{equation*}
%and $\omega (n)^4\leq n^{1/2}p$, by assumption. 
%But 
%$$ \frac{\omega (n)^4}{ (np)^{1/2}} \leq \frac{n^{1/2}p}{(np)^{1/2}}= p^{1/2}.$$

We begin with the denominator of~\eqref{eq:probability}. Let $H(x) = - x \ln x - (1-x) \ln (1-x)$ denote the entropy function defined for $x \in (0,1)$. 
We recall a standard estimate for this binomial coefficient 
which relies on the Stirling approximation for the factorial:
$n! = \sqrt{2\pi n} n^n e^{-n} (1+ O(1/n))$. 
Using this, we write 
\begin{eqnarray*}
&&\binom{n}{n(v)} = \frac{n!}{n(v)! (n-n(v))!} \\
&=& \frac{1+O(1/n(v))}{\sqrt{2\pi}} \cdot \sqrt{\frac{n}{n(v) (n-n(v))}} \cdot \frac{n^n e^{-n}}{n(v)^{n(v)} e^{-n(v)} (n-n(v))^{n-n(v)} e^{-(n-n(v))}} \\
&=& \frac{1+O(1/np)}{\sqrt{2\pi}} \cdot \sqrt{\frac{n}{n(v) (n-n(v))}} \cdot e^{n H\left(n(v)/n\right)}.
\end{eqnarray*}
%As $|n(v) - np|\leq \log n \cdot (np)^{1/2}$ we have 
%$$\sqrt{\frac{1}{1- n(v)/n}} = \sqrt{\frac{1}{1-p}} \left(1 + O\left(\log n\sqrt{\frac{p}{n}}\right) \right) = 
%\sqrt{\frac{1}{1-p}} (1+ o(p^{1/2})). 
%$$
Therefore, 
\begin{equation}\label{eq:denom_estimate} 
\binom{n}{n(v)} \stackrel{\eqref{eq:est3}}{=} (1+ o(p^{1/2})) \sqrt{\frac{1}{1-n(v)/n}} \cdot 
\frac{1}{\sqrt{2\pi n(v)}}\cdot  e^{n H\left(n(v)/n\right)}.
\end{equation}
Now, we will consider $\binom{n_i}{s_i}$. In fact, we will
express this in terms of $\binom{n_i}{m_i}$. In turn, we will
express the latter using~\eqref{eq:denom_estimate}. 
Recalling that $s_i = m_i +\delta_i$, we write 
\begin{eqnarray*}
\binom{n_i}{s_i}/\binom{n_i}{m_i} = \frac{m_i!}{(m_i+\delta_i)!} \cdot \frac{(n_i-m_i)!}{(n_i-m_i-\delta_i)!}.
\end{eqnarray*}
Suppose first that $\delta_i > 0$. Then 
$$\frac{m_i!}{(m_i+\delta_i)!} = \frac{1}{m_i^{\delta_i}} \prod_{j=1}^{\delta_i} \left(1 + \frac{j}{m_i}\right)^{-1}. $$
Writing $1 + j/m_i = \exp \left( \ln (1+ j/m_i)\right)$ and 
expanding the $\ln$ function around 1, we get that 
$$\prod_{j=1}^{\delta_i} \left(1 + \frac{j}{m_i}\right)^{-1} 
=\exp \left( - \frac{\delta_i^2}{2m_i} + O\left(\frac{|\delta_i|}{m_i} + \frac{|\delta_i|^3}{m_i^2} \right)\right).
$$ 
Similarly, we get
$$\frac{(n_i-m_i)!}{(n_i-m_i-\delta_i)!} = 
(n_i-m_i)^{\delta_i} \exp \left(-\frac{\delta_i^2}{2(n_i-m_i)} + O\left(\frac{|\delta_i|}{n_i-m_i} + \frac{|\delta_i|^3}{(n_i-m_i)^2} \right)
\right).
$$
%As $|\delta_i| \leq \omega (n) m_i^{1/2}$, we have that
%$$\frac{|\delta_i|^3}{m_i^2} \leq  \frac{\omega(n)^3}{m_i^{1/2}} \ \mbox{and} \ \frac{|\delta_i|}{m_i} \leq  \frac{\omega(n)}{m_i^{1/2}}.$$
%Hence, 
%$$\frac{|\delta_i|}{m_i} + \frac{|\delta_i|^3}{m_i^2}= 
%O \left(\frac{\omega (n)^3}{m_i^{1/2}} \right) = O\left(\frac{\omega(n)^3}{(np)^{1/2}}\right)\stackrel{\eqref{eq:est2}}{=} o(p^{1/2}).$$
As $n_i-m_i = \Omega (n)$, we deduce that 
$$  \frac{|\delta_i|}{n_i-m_i} + \frac{|\delta_i|^3}{(n_i-m_i)^2} = O \left( \frac{\log n \cdot (np)^{1/2}}{n} + \frac{\log^3 n \cdot (np)^{3/2}}{n^2}\right) =o(p^{1/2}).$$ 
%But $1/m_i^{1/2} = O(1/\sqrt{np})$. But $p = \omega^4 (n) n^{-1/2}$ and thereby $p^{1/2} \cdot \sqrt{np} = n^{1/2} p = \omega^4 (n) \to \infty$ as $n\to \infty$.  
%Thus, 
%$$\frac{|\delta_i|}{m_i} + \frac{|\delta_i|^3}{m_i^2}, \frac{|\delta_i|}{n_i-m_i} + \frac{|\delta_i|^3}{(n_i-m_i)^2}  = o(p^{1/2}), $$
%uniformly over $\delta_i$ such that $|\delta_i| \leq \omega (n) m_i^{1/2}$, for $i=0,1$.
Furthermore, as $|\delta_i| \leq (n(v))^{1/2}\log{n} \stackrel{\Ncal^*_v}{=} O((np)^{1/2} \log{n})$,
\[
\frac{\sum_{i=1}^{k-1}|\delta_i|}{np}+\frac{\sum_{i=1}^{k-1}|\delta_i|^3}{(np)^2} = O\left(\frac{\log^{3}{n}}{\sqrt{np}}\right) \stackrel{\eqref{eq:est1}}{=} o(p^{1/2}).
\]

We thus conclude that 
\begin{equation} \label{eq:num_coeff}
\begin{split}
   & \binom{n_i}{s_i} = \left(1+o(p^{1/2})\right) 
  \binom{n_i}{m_i} \cdot\left(\frac{n_i-m_i}{m_i} \right)^{\delta_i} \cdot \exp\left(-\frac{\delta_i^2}{2} \cdot \frac{n_i}{m_i (n_i - m_i)}\right).
 \end{split}
\end{equation}
Note that as $m_i =\lceil n(v) p_i\rceil = n(v)p_i+r$, for some $0\leq r <1$, and $n_i = np_i$, we have 
\begin{equation} \label{eq:mini_approx}
\frac{n_i}{m_i} = \frac{np_i}{n(v)p_i +r}= \frac{n}{n(v)}\left(1 + O\left( \frac{1}{n(v)}\right)\right).
\end{equation}
%and 
%\begin{equation} \label{eq:mi_est}
%m_i = n(v)p_i(1 + O(1/np)) \stackrel{\eqref{eq:est1}}{=} n(v)p_i(1 + o(p^{1/2})). 
%\end{equation}
Since $|\delta_i| =O( \log n \cdot  (np)^{1/2})$,
we deduce that 
\begin{equation*}
\begin{split}
\left(\frac{n_i-m_i}{m_i} \right)^{\delta_i} &= \left(\frac{n}{n(v)}-1 +O\left(\frac{n}{n(v)^2}\right) \right)^{\delta_i}
\\
& =  \left(\frac{n}{n(v)}-1 \right)^{\delta_i} \left(1+ O\left(\frac{|\delta_i|}{n(v)}\right)\right)\\
& = \left(\frac{n}{n(v)}-1 \right)^{\delta_i} \left(1+ O\left(\frac{\log n}{(np)^{3/2}}\right)\right)\\
& \stackrel{\eqref{eq:est3}}{=}  \left(\frac{n}{n(v)}-1 \right)^{\delta_i} (1 +o(p^{1/2})).
\end{split}
\end{equation*}
Furthermore, 
$$\frac{\delta_i^2}{2} \cdot \frac{n_i}{m_i (n_i - m_i)} =
\frac{\delta_i^2}{2 n(v)p_i} \left(\frac{1}{1-n(v)/n}\right) + O(\delta_i^2/(np)^2). $$
But $\delta_i^2 =O( \log^2 n \cdot (np))$ and therefore
$$\frac{\delta_i^2}{2} \cdot \frac{n_i}{m_i (n_i - m_i)} =
\frac{\delta_i^2}{2 n(v)p_i} \left(\frac{1}{1-n(v)/n}\right) + O(\log^2 n/(np)). $$
By~\eqref{eq:est3}, $\log^2 n /(np) = o(p^{1/2})$ and, thereby,
$$\exp\left(-\frac{\delta_i^2}{2} \cdot \frac{n_i}{m_i (n_i - m_i)}\right) = 
(1+o(p^{1/2})) 
\exp \left(-\frac{\delta_i^2}{2 n(v)p_i} \left(\frac{1}{1-n(v)/n}\right)  \right). $$ 
As $n_i = \Theta(n)$ and $m_i = \Theta (np)$, applying~\eqref{eq:denom_estimate}, we can also deduce that \begin{equation} \label{eq:num_coeff_I}
    \binom{n_i}{m_i} = (1+o(p^{1/2})) \sqrt{\frac{1}{1-n(v)/n}}\cdot 
    \frac{1}{\sqrt{2\pi m_i}} \cdot e^{n_i H(m_i/n_i)}.
\end{equation}
Combining~\eqref{eq:denom_estimate},
~\eqref{eq:num_coeff} and~\eqref{eq:num_coeff_I}, we
deduce that \begin{eqnarray}
  \lefteqn{\prb{(n_1 (v),\ldots, n_k(v)) = (s_1,\ldots, s_{k})\mid n, n(v),\cv^*}=} \nonumber\\
  && \left(1+o(p^{1/2})\right) \times \nonumber \\
  &&\frac{1}{(1-n(v)/n)^{\frac{k-1}{2}}} \cdot 
  \frac{1}{(2\pi)^{\frac{k-1}{2}}} \cdot \sqrt{\frac{n(v)}{\prod_{i=1}^k m_i}} \cdot 
\left(  \prod_{i=1}^k \left( \frac{n}{n(v)}-1 \right)^{\delta_i}
  \right) \times\nonumber \\
  &&
  \hspace{2cm} \exp\left(- \frac{1}{2} \sum_{i=1}^k \frac{\delta_i^2}{n(v)} \frac{1}{p_i(1-n(v)/n)} \right) \times \nonumber \\
  && e^{\sum_{i=1}^k n_i H(m_i/n_i) - n H(n(v)/n)} \nonumber \\
  &\stackrel{\delta_1 + \cdots + \delta_k=0}{=}& \left(1+o(p^{1/2})\right) \times \nonumber \\ 
  &&\frac{1}{(1-n(v)/n)^{\frac{k-1}{2}}}
  \cdot \frac{1}{(2\pi n(v))^{\frac{k-1}{2}}} \cdot \sqrt{\frac{1}{\prod_{i=1}^k p_i}} \times \nonumber \\
 &&  \exp\left(- \frac{1}{2} \sum_{i=1}^k \frac{\delta_i^2}{n(v)} \frac{1}{p_i(1-n(v)/n)} \right) \times \nonumber \\
  &&
  \hspace{2cm} e^{\sum_{i=1}^k n_i H(m_i/n_i) - n H(n(v)/n) }. \label{eq:prob_inter}
\end{eqnarray}

\begin{claim}  \label{clm:exponential}
With $\x = (\delta_1, \ldots,  \delta_{k-1})^T /n(v)^{1/2}$ 
and
\[
\Sigma= \frac{k}{1-n(v)/n}\left[ \begin{array}{ccc} 
2 &\cdots & 1 \\
\vdots & \ddots & \vdots \\
1 & \cdots & 2
\end{array} \right], 
\]
we have
\begin{equation*}
\exp\left(- \frac{1}{2} \sum_{i=1}^k \frac{\delta_i^2}{n(v)} \frac{1}{p_i(1-n(v)/n)} \right)  =(1+o(p^{1/2})) \exp \left(-\frac{1}{2} \boldsymbol{x}^{T} \Sigma \boldsymbol{x} \right).
\end{equation*}
\end{claim}
\begin{proof}
Since $\delta_1 + \cdots + \delta_k=0$ 
we write, with $\boldsymbol{\delta} = (\delta_1,\ldots,\delta_{k-1})^T$,
\begin{eqnarray*}
\sum_{i=1}^k \frac{\delta_i^2}{p_i} &=& 
\sum_{i=1}^{k-1} \frac{\delta_i^2}{p_i}  + \frac{(\delta_1+ \cdots + \delta_{k-1})^2}{p_k} \\
&=& \sum_{i=1}^{k-1} \delta_i^2\left(\frac{1}{p_i} + \frac{1}{p_k}\right)  + \sum_{i\not =j} \delta_i \delta_j 
\frac{1}{p_k} \\
&=& \boldsymbol{\delta}^T 
\left[ \begin{array}{ccc} 
\frac{1}{p_1}+ \frac{1}{p_k} &\cdots & \frac{1}{p_k} \\
\vdots & \ddots & \vdots \\
\frac{1}{p_k} & \cdots &\frac{1}{p_{k-1}} + \frac{1}{p_k}
\end{array} \right] 
\boldsymbol{\delta}.
\end{eqnarray*}
Thus, setting 
\[
\Sigma'= \frac{1}{1-n(v)/n}\left[ \begin{array}{ccc} 
\frac{1}{p_1}+ \frac{1}{p_k} &\cdots & \frac{1}{p_k} \\
\vdots & \ddots & \vdots \\
\frac{1}{p_k} & \cdots &\frac{1}{p_{k-1}} + \frac{1}{p_k}
\end{array} \right],
\]
and using the scaling 
$\x = (\delta_1,\ldots,  \delta_{k-1})^T /n(v)^{1/2}$ we get
\begin{equation*}
\exp\left(- \frac{1}{2} \sum_{i=1}^k \frac{\delta_i^2}{n(v)} \frac{1}{p_i(1-n(v)/n)} \right)  = \exp \left(-\frac{1}{2} \boldsymbol{x}^{T} \Sigma' \boldsymbol{x} \right).
\end{equation*}
Now, note that $p_i = 1/k + \Theta (\log n\cdot n^{-1/2})$. Therefore, uniformly for $\boldsymbol{x}$ such that 
$\| \boldsymbol{x} \|_\infty \leq \log n$, we have 
$\boldsymbol{x}^{T} \Sigma' \boldsymbol{x}  = \boldsymbol{x}^{T} \Sigma \boldsymbol{x}  + O(\log^3 n/n^{1/2}) 
=  \boldsymbol{x}^{T} \Sigma \boldsymbol{x}  + o( p^{1/2} )$. 
The claim then follows.
\end{proof}
\noindent
We now calculate $|\Sigma|$. 
\begin{claim}
We have 
$$|\Sigma| =\frac{k^k}{\left(1-n(v)/n\right)^{k-1}}.   $$
\end{claim}
\begin{proof}
We will consider first the matrix $\un_{k-1} + \mathbf{1}_{(k-1)\times (k-1)}$. We will show that 
\begin{equation}\label{eq:det_sigma'}|\un_{k-1} + \mathbf{1}_{(k-1)\times (k-1)}| = k.
\end{equation}
We will show this by finding the eigenvalues of $\un_{k-1} + \mathbf{1}_{(k-1)\times (k-1)}$ and using the fact that 
$|\un_{k-1} + \mathbf{1}_{(k-1)\times (k-1)}|$ is equal to their product. 
The eigenvectors of $\un_{k-1} + \mathbf{1}_{(k-1)\times (k-1)}$ are $\mathbf{1}_{k-1}$ (the all-ones vector), and the 
vectors $e_{1j}$, for $1< j \leq k-1$ with 1 in the first component and $-1$ at the $j$th 
component and $0$ everywhere else. The associated eigenvalues are $k$ for $\mathbf{1}_{k-1}$ and $1$ for the eigenvectors $e_{1j}$, for $j=2,\ldots, k-1$. 
Thus, the product of the eigenvalues is equal to $k$ whereby~\eqref{eq:det_sigma'} follows.
Therefore, 
$$|\Sigma | = \left( \frac{k}{1-n(v)/n}\right)^{k-1} \cdot |\un_{k-1} + \mathbf{1}_{(k-1)\times (k-1)}| = 
\frac{k^{k}}{\left(1-n(v)/n\right)^{k-1}}. $$

\begin{comment}
We have 
\begin{eqnarray}
|\Sigma| = \frac{1}{(1-n(v)/n)^2} \cdot \left( 
\left(\frac{1}{p_0}+ \frac{1}{p_2} \right) 
\left(\frac{1}{p_1}+ \frac{1}{p_2} \right)-\frac{1}{p_2^2}\right). \nonumber 
\end{eqnarray}
But 
\begin{eqnarray*} 
\left(\frac{1}{p_0}+ \frac{1}{p_2} \right) 
\left(\frac{1}{p_1}+ \frac{1}{p_2} \right)-\frac{1}{p_2^2} 
&=& \frac{1}{p_0p_1} + \frac{1}{p_0p_2} + \frac{1}{p_1p_2} +\frac{1}{p_2^2}-\frac{1}{p_2^2} \\
&=& \frac{1}{p_0p_1} + \frac{1}{p_0p_2} + \frac{1}{p_1p_2} \\
&=& \frac{p_2 + p_1 + p_0}{p_0p_1p_2} = \frac{1}{p_0p_1p_2}.
\end{eqnarray*}
We thus conclude that 
$$ |\Sigma| = \frac{1}{(1-n(v)/n)^2} \cdot \frac{1}{p_0p_1p_2}. $$
\end{comment}
\end{proof}
Returning to~\eqref{eq:prob_inter}, we will show that the 
term 
$$ \left(\frac{1}{1-n(v)/n}\right)^{\frac{k-1}{2}}\sqrt{\frac{1}{p_1\cdots p_{k}}} = (1+o(p^{1/2}))
|\Sigma|^{1/2}.$$
By the above claim, it suffices to show that 
$$\sqrt{\frac{1}{p_1\cdots p_{k}}} = (1+o(p^{1/2})) k^{k/2}. $$
Recall that $p_i = 1/k + O(\log n \cdot n^{-1/2})$, for $i=1,\ldots, k$. 
Therefore,
$$\sqrt{\frac{1}{p_1\cdots p_{k}}} = k^{k/2} \left(1+O(\log n \cdot n^{-1/2})\right) = k^{k/2} \left(1+o(p^{1/2}))\right), $$
as $p \gg \log^2  n/n$.

\noindent
Finally, we will deal with last exponential in~\eqref{eq:prob_inter}.
\begin{claim}
We have 
$$\sum_{i=1}^k n_i H(m_i/n_i) - n H(n(v)/n)   = O(1/np). $$ \qedhere
\end{claim}
\begin{proof}
We will approximate $H(m_i/n_i)$ by $H(n(v)/n)$ using (the second order) Taylor's theorem around
$n(v)/n$. Let $i \in \{1,\ldots, k-1\}$. Then 
$m_i \geq n(v) p_i$, whereby $n(v)/n \leq m_i /n_i$. Furthermore, for $n$ sufficiently large $m_i/n_i<1$ and therefore $[n(v)/n,m_i/n_i] \subset (0,1)$. 
Since the entropy function $H$ is twice differentiable in $(0,1)$,  
there exists $\xi \in [n(v)/n,m_i/n_i]$
$$H\left(\frac{m_i}{n_i}\right)= H\left(\frac{n(v)}{n_i}\right) + H'\left(\frac{n(v)}{n}\right)
\left(\frac{m_in - n_in(v)}{nn_i}\right) 
+ \frac{1}{2} H''(\xi)\left(\frac{m_in - n_in(v)}{nn_i}\right)^2.$$
By~\eqref{eq:mini_approx}, we have that
$$\left(\frac{m_in - n_in(v)}{nn_i}\right)^2 = O\left( \frac{1}{n^2} \right). $$
Also, $H'' (x) = - \frac{1}{x(1-x)}$, for any $x\in (0,1)$. As $\xi = p +o(p)$ and 
$p$ is asymptotically bounded away from 1, we have 
$$|H''(\xi)| = \frac{1}{\xi (1-\xi)} = O(1/\xi) = O(1/p).$$
We can deduce the same for $i=k$, except that in that case
$m_k/n_k \leq n(v)/n$. 
Using these, we can write
\begin{eqnarray}
\lefteqn{\sum_{i=1}^k n_iH\left(\frac{m_i}{n_i}\right)=} \nonumber\\
&&\sum_{i=1}^k n_i H\left(\frac{n(v)}{n}\right) 
+ \sum_{i=1}^k n_i H'\left(\frac{n(v)}{n}\right)
\left(\frac{m_in - n_in(v)}{nn_i}\right) + 
O\left(\frac{1}{n^2 p}\right)\sum_{i=1}^k n_i \nonumber \\
&=& n H\left(\frac{n(v)}{n}\right)  
+  H'\left(\frac{n(v)}{n}\right)
\left(\frac{n\sum_i m_i - n(v) \sum_i n_i}{n}\right) 
+O\left(\frac{1}{np}\right) \nonumber \\
&=& n H\left(\frac{n(v)}{n}\right)  
+O\left(\frac{1}{np}\right), \nonumber
\end{eqnarray}
since $\sum_i m_i = n(v)$ and $\sum_i n_i =n$. 
\end{proof}
Therefore, 
\begin{equation} \label{eq:last_term}
e^{\sum_{i=1}^k n_i H(m_i/n_i) - n H(n(v)/n)} = e^{O(1/np)} = 1 + O(1/(np)) \stackrel{\eqref{eq:est1}}{=} 1+o(p^{1/2}).
\end{equation}
We thus conclude that
\[
\prb{(n_1 (v),\ldots, n_k (v)) = (s_1,\ldots,s_k)  \mid n, n(v), \cv^*} =\frac{1}{n(v)^{\frac{k-1}{2}}} \cdot \phi(\boldsymbol{x})\cdot\left(1+o(p^{1/2})\right).  \qedhere
\]
\end{proof}
\subsection{Some basic (anti-)concentration properties of the initial configuration} \label{sec:anticonc}
In this section we will prove that with high probability the random variables $n_i$, for $i \in \Mcal_0$, are not too close to or too far from $n^*/k_0$. Moreover, we show that it is also unlikely for the $c_i$s to be too close to each other. 
\begin{lemma} \label{lem:anti_conc}
For any $\eps >0$ there exists $\delta>0$ such that 
with probability at least $1-\eps$, for $i\in\Mcal_0$ we have 
\begin{equation} \label{eq:event_1} \delta < \frac{1}{\sqrt{n^*}} |n_i - n^*/k_0|\leq 1/\delta, \end{equation}
and for distinct $i=1, \ldots, k_0-1$ 
\begin{equation} \label{eq:event_2} |c_i -c_{i+1}|> \delta, \end{equation}
with $c_{k_0} = -\sum_{j<k_0}c_j$.
\end{lemma}
\noindent
We denote the event that is described by~\eqref{eq:event_1} and \eqref{eq:event_2} by $\Ev_{\delta, n}$.
\begin{proof}
Conditional on $|S^*|=n^*$, for any $i \in \Mcal_0$ we have $n_i \sim \bin (n^*,1/k_0)$. The local limit theorem (Theorem~\ref{thm:LLT}) implies that the probability that 
$n_i$ is too close to $n^*/k_0$ is small. 
\begin{lemma} \label{lem:anti_conc}
For any $\eps>0$ there exists $\delta > 0$ such that 
$$\prb{\left| n_i - n^* /k_0 \right|\leq \delta \sqrt{n^*} \mid n^*} < \eps. $$
\end{lemma}
\begin{proof}
We use the local limit theorem (Theorem~\ref{thm:LLT}). 
We denote the variance of $n_i$ (conditional on $n^*$)  by 
$\sigma_*^2 = n^* (1/k_0)(1- 1/k_0)$ and write
\begin{eqnarray*}
\prb{|n_i- n^*/k_0 | \leq \delta \sqrt{n^*} \mid n^*} &=& \sum_{k: |k-  n^*/k_0|\leq \delta \sqrt{n^*}} 
\prb{n_i = k}  \\
&=& \frac{1}{\sqrt{2\pi \sigma_*^2}} \sum_{k: |k- n^*/k_0 |\leq \delta \sqrt{n^*}} e^{- \frac{(k-n^*/k_0)^2}{2\sigma_*}} + O\left(\frac{1}{\sqrt{n^*}}\right) \\
&\leq& \frac{2\delta \sqrt{n^*}}{\sqrt{2\pi\sigma_*^2}} + O\left(\frac{1}{\sqrt{n^*}}\right) = O(\delta).
\end{eqnarray*}
The lemma follows.
\end{proof}
Also, the Chernoff bound~\eqref{eq:Chernoff} implies that $n_i$ is not too far 
from $n^*/k_0$ too. 
\begin{lemma} \label{lem:conc}
For any $\eps>0$, there exists $C>0$ such that 
$$\prb{|n_i - n^*/k_0| > C \sqrt{n^*} \mid n^*} < \eps/6. $$
\end{lemma}
Furthermore, we will show the following.
\begin{lemma} \label{lem:cstar_dot}Let $X = (X_1,\ldots, X_{k_0-1})$ be a random vector with
$X \sim \mathrm{Mult} (n^*;(1/k_0,\ldots, 1/k_0))$ and let
$\cv^* =(n^*)^{-1/2} \cdot (X- n^* (1/k_0,\ldots, 1/k_0))$.
Let $\boldsymbol{C}\in \mathbb{R}^{k_0-1}$ be such that $\boldsymbol{C}\not = \boldsymbol{0} \in \mathbb{R}^{k_0-1}$. 
For any $\eps >0$, there exists $\delta = \delta (\eps, \boldsymbol{C})>0$ such that 
$$\prb{|\cv^* \cdot \boldsymbol{C} |<\delta \mid n^*} < \eps,$$
for any $n^*$ that is sufficiently large. 
\end{lemma}
\begin{proof}
For $j=1,\ldots, k_0-1$, the component $X_j$  is binomially distributed with parameters 
$n^*$ and $1/k_0$.  The same holds for $X_{k_0} = n^* - \sum_{j<k_0}X_j$. 
Thus, for any $\eps >0$ and any $j\in [k_0]$, there is $\Lambda_j =\Lambda_j (\eps)>0$, such that 
$$\prb { (n^*)^{-1/2} \cdot \left|X_j - n^*/k_0 \right| > \Lambda_j } < \eps /k_0. $$ 
%For each $j\in [k_0]$, we set $c_j= (n^*)^{-1/2} \cdot (X_j - n^*/k_0 )$ and $\cv^* = (c_1,\ldots, c_{k_0-1})$.
By the union bound, for any $\eps$ there is $\Lambda = \max \{\Lambda_j (\eps/3)\}_{j\in [k_0]}$, such that 
\begin{equation}\label{eq:cv_infty_norm} \prb{\left\| \cv^* \right\|_{\infty} > \Lambda} < \eps/3. 
\end{equation}
For a given $\eps>0$, select this $\Lambda$.

%Now, let $D_\Lambda :=\{(n_1,\ldots, n_{k_0-1}) \ : \ n_j \in \mathbb{N}_0, n^{-1/2} \cdot |n_j - n\lambda_j | \leq \Lambda, \mbox{for $j=1,\ldots, k_0-1$} \}$. 
We will use the local limit theorem for the multinomial distribution. This was first shown by Arenbaev~\cite{ar:Arenbaev76} (see Lemma 2 therein), but we will use a version that was shown by Siotani and Fujikoshi in~\cite{ar:SiotFuji84}. 
Let $X=(X_1,\ldots, X_{k_0-1})$ be the random vector that follows the multinomial distribution with parameters 
$n^*$ and $\boldsymbol{q} := (1/k_0,\ldots, 1/k_0) \in \mathbb{R}^{1\times (k_0-1)}$.
 Let $\phi $ denote the pdf of the $k_0-1$-dimensional normal distribution with expected value equal to $\boldsymbol{0}\in \mathbb{R}^{k_0-1}$ and covariance matrix $\mathrm{diag} (1/k_0,\ldots, 1/k_0) 
 - \boldsymbol{q}^T \boldsymbol{q}$. 
The Local Limit Theorem for the multinomial distribution implies that  for any 
$\x = (x_1,\ldots, x_{k_0-1})$ such that $\|\x\|_{\infty} \leq \Lambda$:
\begin{equation} \label{eq:LLL} 
\prb{\cv^*=\x} \leq \frac{1}{(n^*)^{\frac{k_0-1}{2}}} \phi(\x) (1+ O((n^*)^{-1/2})).
\end{equation}

For a vector $\boldsymbol{C}= (C_1,\ldots ,C_{k_0-1})^T \in \mathbb{R}^{k_0-1}$, we define the set 
$$S_{\delta, \Lambda,\boldsymbol{C}}=\{\x = (x_1,\ldots, x_{k_0-1}) :  |\x \cdot \boldsymbol{C}  | < \delta,  \left\| \x\right\|_\infty \leq \Lambda, \ \mbox{$\forall j\in [k_0-1]$} \ n^* /k_0 +x_j \sqrt{n^*} \in \mathbb{N}_0\ \}.$$
We claim that $|S_{\delta, \Lambda, \boldsymbol{C}}| = O(\delta\cdot (n^*)^{\frac{k_0-1}{2}} \cdot \Lambda^{k_0-1})$. 
Indeed, we can select freely $x_1,\ldots, x_{k_0-2}$, whereby there are $\left( 2\Lambda \sqrt{n^*} \right)^{k_0-2}$ such choices. 
For each such choice, the component $x_{k_0-1}$ must be selected so that $- \delta < \sum_{j=1}^{k_0-2}x_j C_j + x_{k_0-1}C_{k_0-1} < \delta$. 
There are at most $2\delta \sqrt{n^*} \Lambda / |C_{k_0-1}|$ such choices.
We conclude that 
$$|S_{\delta, \Lambda, \boldsymbol{C}}| \leq\delta \cdot (n^*)^{\frac{k_0-1}{2}} \cdot (2 \Lambda)^{k_0-1}/|C_{k_0-1}|. $$
Now, choose $\delta = \frac{1}{2} \eps  (2\Lambda)^{-(k_0-1)}\cdot |C_{k_0-1}|$; this implies 
that 
\begin{equation} \label{eq:S_bound} 
|S_{\delta, \Lambda, \boldsymbol{C}}| \leq \frac{\eps}{2} \cdot (n^*)^{\frac{k_0-1}{2}}.
\end{equation} 
Therefore,
\begin{eqnarray*}
\prb{|\cv^* \cdot \boldsymbol{C} | < \delta} &\stackrel{\eqref{eq:LLL}}{\leq}& (1+o(1)) \frac{1}{(n^*)^{\frac{k_0-1}{2}}} \sum_{\x \in S_{\delta, \Lambda,\boldsymbol{C}}}\phi (\x) + \prb {\left\| \cv^* \right\|_{\infty} > \Lambda} \\ 
&\leq &(1+o(1)) \frac{1}{(n^*)^{\frac{k_0-1}{2}}} |S_{\delta, \Lambda, \boldsymbol{C}}| + \prb {\left\| \cv^* \right\|_{\infty} > \Lambda} \\
&\stackrel{\eqref{eq:cv_infty_norm},\eqref{eq:S_bound}}{\leq}&(1+o(1))\eps/2+\eps/3 < \eps.
\end{eqnarray*}
\end{proof} 

The two lemmas together with the union bound imply that for any $\eps$ there exists $\delta>0$ such that 
with probability at least $1-\eps/2$, for $i\in\Mcal_0$ we have 
\begin{equation*}  \delta < \frac{1}{\sqrt{n^*}} |n_i - n^*/k_0|\leq 1/\delta. \end{equation*}
Now,  applying Lemma~\ref{lem:cstar_dot} with $\boldsymbol{C}=(0,\ldots, 0, 1,-1,0,\ldots,0)^T$, where $1$ and $-1$ are in places $i$ and $i+1$, respectively, with $i<k_0-1$, 
or taking $\boldsymbol{C}=(1,\ldots, 1, 2)^T$, we deduce that 
 for $i=1, \ldots, k_0-1$ 
\begin{equation*} |c_i -c_{i+1}|> \delta, \end{equation*}
with $c_{k_0} = -\sum_{j<k_0}c_j$ with probability at least $1- \eps /2$. The union bound implies the statement of 
the lemma.
\end{proof}

In what follows, we will condition on $\Ev_{\delta, n}$. 
More specifically, we shall assume that the initial configuration on $S^*$ is such that for each $i \in \Mcal_0$
$$n_i = n^*/k_0 + c_i \sqrt{n^*}, $$ with $\delta <c_i < 1/\delta$ and $c_{i+1} -c_i > \delta$, for any $i=1,\ldots, k_0-1$. 
We write $S_0 \in \Ev_{\delta, n}$.

\section{The configuration after one round} 
\subsection{Estimating the probability of a certain state after one round} \label{sec:exp_round_one}
For $i \in \Mcal_0$, we let $\Scal\Scal_1^{(i)}$ denote the subset of vertices $v \in V_n$ such that $S_1(v)=i$ in a \emph{strong sense}, that is,  $n_i(v) > n_j (v)$ for any $j\in \Mcal_0$ with $j \not = i$. 
For a vertex $v\in V_n$, we will give a tight approximation of the probability that 
$v\in \Scal\Scal_1^{(i)}$, for $i \in \Mcal_0$.  
Our aim is to show the following result. 
\begin{lemma} \label{lem:conclusion}
Let $v\in V_n$. There is a constant $\Lambda >0$ such that for $n$ sufficiently large, on the events 
$S_0\in \Ev_0 \cap \Ev_{\delta, n}$ and $\Ncal^*_{v}$ and
assuming that $\cv^*$ is such that $c_1 > c_2 > \cdots > c_{k_0}$, for $i=1,\ldots, k_0-1$: 
\begin{equation} \label{eq:conclusion}
\prb{v \in \Scal \Scal_1^{(i)} \mid  n^*(v), n_v^*, \cv^*} - \prb{v \in \Scal \Scal_1^{(i+1)} \mid n^*(v), n_v^*, \cv^*} >\Lambda \cdot (c_i- c_{i+1}) p^{1/2}.
\end{equation}
\end{lemma}

\begin{proof}
%Let $\Mcal_0 = \argmax \{  \lambda_i \ : \ i \in [k] \}$. 
%
%
%Previously we have shown that for any $j \not \in \Mcal$ the probability that $S_1(v)=j$ is $o(1/n)$; this is given by Theorem \ref{cor:2Dropout} and Corollary \ref{cor:2Dropout}. Hence, by the union bound, the probability
%that there is a vertex $v \in V_n$ that has $S_1(v) \not \in \Mcal$ is equal to $o(1)$. 
%
%For the remainder of this section, we condition on this event for all $v \in V_n.$ We will estimate the probability of the event $S_1(v)=i$, for $i \in \Mcal$. 
%%However, we will show that one of them is larger than $1/2$ by a term that is of order $p^{1/2}$. Which one that is depends on the initial configuration of $\mathcal{S}_0$.
%%In other words, with probability $1-o(1)$ all vertices will play Strategy 0 or 1 after the first step. 
%%Thus, we need to determine the probability of each of them. We will argue as in Lemma \ref{lem:ScoreComparison}; however, 

%We intend to use Lemma~\ref{lem:llt}. We will take $t=0$ or $=1$, depending on whether $v\not \in S^*$ or 
%$v\in S^*$, respectively. To ease notation, we will denote by $n(v)$ the number of neighbours of $v$ in $S^*\setminus \{v\}$, instead of $\hat{n}^*(v)$ which we used in Section~\ref{sec:LLT}. 

Recall that $p_i = 1/k_0 + \frac{c_i}{\sqrt{n^*}} + O(1/n)$ and we wrote 
$n_i (v) = m_i + \delta_i$, where 
$m_i = \lceil n^*(v)p_i \rceil$, for $i=1,\ldots, k_0-1$, but
$m_{k_0} = n^*(v) - \sum_{i=1}^{k-1} m_i$. 

%With this notation, we write for $i\in \Mcal$
%\begin{equation*}
%    T_0 (v;i) = \sum_{j=1}^k (m_j + \delta_j) q_{i,j}.
%\end{equation*}
%But $\delta_k = - \sum_{j=1}^{k-1} \delta_j$, whereby
%\begin{equation*}
%    T_0 (v;i) = \sum_{j=1}^{k }m_j  q_{i,j}
%    +\sum_{j=1}^{k-1} \delta_j (q_{i,j} - q_{i,k}).
%\end{equation*}
Now, we write $m_i = n^*(v)p_i + r_i$, for 
$0\leq r_i <1$, and using that $p_i = 1/k_0+ \frac{c_i}{\sqrt{n^*}} + O(1/n)$, we have that 
\begin{eqnarray*}    
m_i  &=& 
n^*(v) /k_0 + \frac{n^*(v)}{\sqrt{n^*}} c_i +r_i + O\left(\frac{n^*(v)}{n^*}\right) \\
&=:& n^*(v)/k_0 + h_i(\cv^*, n^*(v)).
\end{eqnarray*}
In other words, we have set $h_i (\cv^*, n^*(v)) :=m_i - 
n^*(v) /k_0$. Note that $\sum_{i=1}^{k_0} h_i (\cv^*, n^*(v)) =0$ and, moreover, 
\begin{equation*}
n_i(v) =n^*(v) /k_0 + h_i(\cv^*, n^*(v)) +\delta_i. 
\end{equation*}
Rescaling by $n^*(v)^{1/2}$ and recalling that $\boldsymbol{x} = \boldsymbol{\delta} / n^*(v)^{1/2}$ we get 
\begin{eqnarray*}
n^*(v)^{-1/2} n_i (v) &=& n^*(v)^{1/2} /k_0 + \frac{h_i (\cv^*, n^*(v))}{n^*(v)^{1/2}} + \frac{\delta_i}{n^*(v)^{1/2}} \\
&=& n^*(v)^{1/2} /k_0 + \sqrt{\frac{n^*(v)}{n^*}} c_i +x_i +  O(n^*(v)^{-1/2}). 
\end{eqnarray*}
Let $\Lcal$ denote the set of points 
$\{ n^*(v)^{-1/2}\cdot (\delta_1,\ldots ,\delta_{k_0-1}): \delta_i \in \mathbb{Z}, i=1,\ldots, k_0-1 \}$ and let $\Lcal'$ denote
its restriction $\{ (x_1,\ldots, x_{k_0-1}) \in \Lcal  \ : \ 
|x_i|\leq \log n, i=1,\ldots, k_0-1 \}$. 
By Claim~\ref{clm:hypergeom_no_large_cs}, conditional on a realisation of $\{S_0^{-1}(1),\ldots, S_0^{-1}(k_0)\}$ that satisfies the event $\Ev_{\delta,n}\cap \Ev_0$, and on the value of $n^*(v)$ we have 
\begin{align}
&\prb{\mbox{for some $i \in \Mcal_0$ }(n^*(v))^{-1/2}|n_i(v)- n^*(v) (n_i/n^*)|> \log n \mid n^*(v), n_v^*,\cv^*} =\nonumber \\
&\hspace{2cm} \exp\left({-\Omega (\log^2 n)}\right) =o(p^{1/2}). \label{eq:outside_L'}
\end{align} 
In other words, realisations of $(n^*(v))^{-1/2} (n_i(v))_{i\in \Mcal_0}$ that do not belong to $\Lcal'$ have probability $o(p^{1/2})$. 
For $i=1,\dots, k_0$, we let $\Dcal_{i} = \{\boldsymbol{x} \in \mathbb{R}^{k_0-1} \ : \  x_i > x_j, \mbox{for all $j=1,\ldots ,k_0$ such that $j\not = i$} \}$ and 
$\Dcal_{i}' = \{\boldsymbol{x} \in \mathbb{R}^{k_0-1} \ : \  x_i + h_i(\cv^*, n^*(v)) > x_j + h_j(\cv^*, n^*(v)), \mbox{for all $j=1,\ldots, k_0$ such that $j \not =i$}\}$, where $x_{k_0} = - \sum_{i <k_0}x_i$.

Using Lemma~\ref{lem:llt} and~\eqref{eq:outside_L'}, we can then write 
\begin{eqnarray} \label{eq:prob_i}
\prb{v \in \Scal \Scal_1^{(i)} \mid n^*(v),n_v^*,  \cv^*} =(1+o(p^{1/2}))\cdot \frac{1}{n^*(v)^{\frac{k_0-1}{2}}} \sum_{\x \in \Dcal_i' \cap \Lcal'} \phi(\x) + o(p^{1/2}). 
\end{eqnarray}

Let us now assume, without loss of generality, that $c_1 > \ldots > c_{k_0}$. 
Recall that if the event $\Ev_{n,\delta}$ is realised, then $c_{i} - c_{i+1} > \delta$  for $i=1,\ldots, k_0-1$. 
(By Lemma~\ref{lem:anti_conc}, this event occurs with probability $1-\eps$ for 
any $\eps>0$ and for some $\delta = \delta (\eps)$ - cf. Section~\ref{sec:anticonc}.)

%Let $q_i = \int_{\Dcal_i'}\phi (\boldsymbol{x}) \mathrm{d} \boldsymbol{x}$. 
We will show the following. 
\begin{lemma} \label{lem:int_decomp} There is $\Lambda >0$ such that for any $n$ sufficiently large, 
on the events $\Pcal_0, \Ev_{\delta,n}$ and $\Ncal_v^*$ 
we have: 
\begin{equation*} \label{eq:phi_perturb} 
n^*(v)^{-\frac{k_0-1}{2}}\cdot \left( \sum_{\x \Dcal_i'\cap \Lcal'}\phi (\x)  -  \sum_{\x \in \Dcal_{i+1}'\cap \Lcal'} \phi (\x) \right) > \Lambda (c_i - c_{i+1}) p^{1/2}.
\end{equation*}
\end{lemma}
\begin{proof}
Let $\h_n = (h_1(\cv^*, n^*(v)),\ldots, h_{k_0-1} (\cv^*, n^*(v)))^T$.
We use the transformation $\y = \x + \h_n$ and write 
\begin{eqnarray*}
\sum_{\x \in \Dcal_i'\cap \Lcal'} \phi (\x)  = \sum_{\y \in \Dcal_i : \ \y-\h_n \in \Lcal'} \phi (\y - \h_n). 
\end{eqnarray*}
Our aim now is to replace $\phi (\y - \h_n)$ in the above sum with $\phi (\y)$. 
To estimate the error, we need to consider the difference $\phi (\y- \h_n) - \phi (\y)$. 
We write: 
\begin{eqnarray*} 
\phi (\y - \h_n)  &=& \phi (\y) \cdot \exp \left( \y^T \Sigma \h_n - \frac{1}{2} \h_n^T \Sigma \h_n \right).
%\\ &=& \phi (\x ) \left( \exp \left( \x^T \Sigma \h_n - \frac{1}{2} \h_n^T \Sigma \h_n \right)  -1 \right).
\end{eqnarray*}
%Therefore, 
%\begin{eqnarray*} 
%\int_{\Dcal_i} \phi (\x - \h_n) \mathrm{d}\x &=& \int_{\Dcal_i} \phi (\x ) \mathrm{d} \x + \\
%&& \hspace{1cm} \int_{\Dcal_i} \phi (\x ) \left(\exp \left( \x^T \Sigma \h_n - \frac{1}{2} \h_n^T \Sigma \h_n \right)  -1  \right) \mathrm{d}\x. 
%\end{eqnarray*}
%By symmetry, for any $i,j=1,\ldots, k_0-1$ we have 
%$$ \int_{\Dcal_i} \phi (\x ) \mathrm{d} \x =\int_{\Dcal_j} \phi (\x ) \mathrm{d} \x. $$
%We denote the common value by 
%$$q:= \int_{\Dcal_i} \phi (\x ) \mathrm{d} \x. $$
%Furthermore, we denote the second integral by
%$$ \beta_i = \int_{\Dcal_i} \phi (\x ) \left(\exp \left( \x^T \Sigma \h_n - \frac{1}{2} \h_n^T \Sigma \h_n \right)  -1  \right) \mathrm{d}\x.
% $$
For $i=1, \ldots, k_0-2$, we will define a bijection between the sets $\Dcal_i$ and $\Dcal_{i+1}$:  
$$\y  = (y_1,\ldots, y_i,y_{i+1},\ldots, y_{k_0-1}) \mapsto 
\hat{\y} = (y_1,\ldots, y_{i+1},y_i,\ldots, y_{k_0-1}) \in \Dcal_{i+1},$$ 
whereas for $i=k_0-1$ we set 
$$\y  = (y_1,\ldots, y_i,y_{i+1},\ldots, y_{k_0-1}) \mapsto 
\hat{\y} = (y_1,\ldots, y_{i+1},y_i,\ldots, -\sum_{i<k_0} y_i) \in \Dcal_{k_0}.$$
\begin{claim} For any $\y \in \Dcal_i$, with $i\in \{1, \ldots, k_0-1\}$, we have 
$$ \y^T \Sigma \y = \hat{\y}^T \Sigma \hat{\y}. $$
\end{claim}
\begin{proof}
Indeed, with $\Sigma' = \left[ \begin{array}{ccc} 
2 &\cdots & 1 \\
\vdots & \ddots & \vdots \\
1 & \cdots & 2
\end{array} \right] $, it suffices to prove that $\y^T \Sigma' \y = \hat{\y}^T \Sigma' \hat{\y}$, as $\Sigma$ is a multiple of $\Sigma'$ by a scalar. 
For $i< k_0-1$, we have 
\begin{eqnarray*}
\y^T \Sigma' \y &=& 2 \sum_{j=1}^{k_0-1} y_i^2 + \sum_{j' \not = j} y_jy_{j'} \\
&=& 2 \sum_{j=1}^{k_0-1} \hat{y}_i^2 + \sum_{j' \not = j} \hat{y}_j \hat{y}_{j'}.
\end{eqnarray*}
For $i=k_0-1$,  
\begin{eqnarray*}
\hat{\y}^T \Sigma' \hat{\y} &=& 2 \sum_{j=1}^{k_0-1} \hat{y}_i^2 + \sum_{j' \not = j} \hat{y}_j \hat{y}_{j'} \\ 
&=&  2 \sum_{j=1}^{k_0-2} y_i^2  +2 \left(\sum_{i<k_0}y_i\right)^2 + \sum_{j' \not = j < k_0-1} \hat{y}_j \hat{y}_{j'} 
+2 \hat{y}_{k_0-1} \sum_{i< k_0-1} y_i \\
&=& 2 \sum_{j=1}^{k_0-2} y_i^2  +2 \left(\sum_{i<k_0}y_i\right)^2 + \sum_{j' \not = j < k_0-1} y_j y_{j'} 
-2 \left(\sum_{i<k_0} y_i\right) \left( \sum_{i< k_0-1} y_i\right) \\
&=& 2 \sum_{j=1}^{k_0-2} y_i^2  +2 \left(\sum_{i<k_0}y_i\right)^2 + \sum_{j' \not = j < k_0-1} y_j y_{j'} 
-2 \left(\sum_{i<k_0} y_i\right) \left( \sum_{i< k_0} y_i  - y_{k_0-1}\right) \\
&=& 2 \sum_{j=1}^{k_0-2} y_i^2  + \sum_{j' \not = j < k_0-1} y_j y_{j'} +2 y_{k_0-1}^2 +2 y_{k_0-1} \sum_{j<k_0-1}y_j\\
&=& 2 \sum_{j=1}^{k_0-1} y_i^2 + \sum_{j' \not = j} y_jy_{j'}  =\y^T \Sigma' \y.
\end{eqnarray*}
\end{proof}
The above claim implies that $\phi (\y) = \phi (\hat{\y} )$, for any $\y \in \Dcal_i$.  
We will use this in bounding from below the difference: 
\begin{eqnarray*} 
 \lefteqn{\sum_{\y \in \Dcal_i : \ \y-\h_n \in \Lcal'} \phi (\y - \h_n) -  \sum_{\y \in \Dcal_{i+1} : \ \y-\h_n \in \Lcal'} \phi (\y - \h_n)=} \\
 &=& \sum_{\y \in \Dcal_i : \ \y-\h_n \in \Lcal'} \phi (\y) e^{\y^T \Sigma \h_n - \frac{1}{2} \h_n^T \Sigma \h_n } -  \sum_{\y \in \Dcal_{i+1} : \ \y-\h_n \in \Lcal'} \phi (\y) e^{\y^T \Sigma \h_n - \frac{1}{2} \h_n^T \Sigma \h_n }.
\end{eqnarray*}
Using the above bijection, the last term can be written as: 
\begin{eqnarray*}
\sum_{\y \in \Dcal_{i+1} : \ \y-\h_n \in \Lcal'} \phi (\y) e^{\y^T \Sigma \h_n - \frac{1}{2} \h_n^T \Sigma \h_n } &=& 
\sum_{\y \in \Dcal_{i} : \ \y-\h_n \in \Lcal'} \phi (\hat{\y}) e^{\hat{\y}^T \Sigma \h_n - \frac{1}{2} \h_n^T \Sigma \h_n }\\
&\stackrel{\phi (\y) = \phi (\hat{\y})}{=}& 
\sum_{\y \in \Dcal_{i} : \ \y-\h_n \in \Lcal'} \phi (\y) e^{\hat{\y}^T \Sigma \h_n - \frac{1}{2} \h_n^T \Sigma \h_n }.
\end{eqnarray*}
Thereby, the difference can be written as:
\begin{eqnarray}
 \lefteqn{\sum_{\y \in \Dcal_i : \ \y-\h_n \in \Lcal'} \phi (\y - \h_n) -  \sum_{\y \in \Dcal_{i+1} : \ \y-\h_n \in \Lcal'} \phi (\y - \h_n)=} \nonumber \\
 &=& \sum_{\y \in \Dcal_i : \ \y-\h_n \in \Lcal'} \phi (\y) e^{\y^T \Sigma \h_n - \frac{1}{2} \h_n^T \Sigma \h_n } - 
 \sum_{\y \in \Dcal_{i} : \ \y-\h_n \in \Lcal'} \phi (\y) e^{\hat{\y}^T \Sigma \h_n - \frac{1}{2} \h_n^T \Sigma \h_n } \nonumber \\
 &=& e^{- \frac{1}{2} \h_n^T \Sigma \h_n}\cdot \sum_{\y \in \Dcal_i : \ \y-\h_n \in \Lcal'} \phi (\y)  
 \left(e^{\y^T \Sigma \h_n} - e^{\hat{\y}^T \Sigma \h_n} \right).  \nonumber \\ 
 & & \label{eq:sum_diff}
\end{eqnarray}
By the mean value theorem, for any $\y \in \Dcal_i$, there is $\xi$ such that 
$\max\{\y^T \Sigma \h_n, \hat{\y}^T \Sigma \h_n\} > \xi > \min \{ \y^T \Sigma \h_n, \hat{\y}^T \Sigma \h_n\}$ and 
\begin{equation} \label{eq:meanvalue_app}
e^{\y^T \Sigma \h_n} - e^{\hat{\y}^T \Sigma \h_n}= e^{\xi}\cdot  \left( (\y -\hat{\y})^T\Sigma \h_n \right).
\end{equation}
Suppose, firstly, that $i<k_0-1$. 
With $\e_i \in \mathbb{R}^{k_0-1}$ being the vector that has 0s everywhere apart from the $i$th component which is equal to 1, we express:
$$\y -\hat{\y} = (y_i - y_{i+1}) (\e_i - \e_{i+1}). $$
As, for each $i=1,\ldots, k_0-1$, 
 $\e_i \Sigma' \h_n =  (1,\ldots, 1,2,1,\ldots, 1) \h_n = \sum_{j=1}^{k_0-1} h_j + h_i$, it follows that 
 $$ (\e_i - \e_{i+1})^T \Sigma' \h_n = h_i-h_{i+1}.$$ 
 Therefore, 
 $$(\y -\hat{\y})^T\Sigma' \h_n= (y_i-y_{i+1}) (h_i - h_{i+1}). $$
 Now, suppose that $i=k_0-1$. 
 In this case and recalling that $\hat{y}_{k_0-1} =- \sum_{j<k_0} y_j = y_{k_0}$ we have 
 $$\y -\hat{\y} =(y_{k_0-1} - \hat{y}_{k_0-1}) \e_{k_0-1} = (y_{k_0-1} - y_{k_0}) \e_{k_0-1}. $$ 
 But $\e_{k_0-1}^T \Sigma' \h_n = \sum_{i=1}^{k_0-1}h_i  + h_{k_0-1}= - h_{k_0} + h_{k_0-1}$.
 Therefore, 
 $$ (\y -\hat{\y})^T\Sigma' \h_n = (y_{k_0-1} - y_{k_0}) (h_{k_0-1} - h_{k_0}).$$
 
 Now, recall that $h_i = \sqrt{\frac{n^*(v)}{n^*}}c_i + o(p^{1/2})$ and
 $n^*(v)/n^* = p (1+o(1))$. These imply that
 $$ h_i -h_{i+1} = p^{1/2}(c_i - c_{i+1}) (1+o(1)).$$ 
 So, uniformly over $\y \in \Dcal_i$,
 $$ (\y -\hat{\y})^T\Sigma \h_n=\frac{k_0}{1-n^*(v)/n_v^*} (y_i-y_{i+1})p^{1/2}(c_i - c_{i+1}) (1+o(1))$$
 and, substituting this into~\eqref{eq:meanvalue_app}, Equation~\eqref{eq:sum_diff} yields: 
 \begin{eqnarray*} 
 \lefteqn{\sum_{\y \in \Dcal_i : \ \y-\h_n \in \Lcal'} \phi (\y - \h_n) -  \sum_{\y \in \Dcal_{i+1} : \ \y-\h_n \in \Lcal'} \phi (\y - \h_n)=} \nonumber \\ 
 &=& \nonumber \frac{k_0}{1-n^*(v)/n_v^*}  (1+o(1)) p^{1/2} (c_i-c_{i+1}) e^{- \frac{1}{2} \h_n^T \Sigma \h_n}\cdot \sum_{\y \in \Dcal_i : \ \y-\h_n \in \Lcal'} \phi (\y)  e^{\xi} 
  (y_i-y_{i+1}).
 \end{eqnarray*}
 Now, for some $0< \zeta < 1$, let $\Dcal_{i}^{(\zeta)} = \{ \y \in \Dcal_i \ : \ y_i -y_{i+1}> \zeta \ \mbox{and} \ \zeta < y_{j} <1/\zeta \ \mbox{for all $j=1,\ldots, k_0-1$} \}$.  Since $\|\cv^* \|_{\infty} < 1/\delta$, there is a $\Xi = \Xi (\zeta, \eps ,\cv^*) >0$ such that $e^\xi > \Xi$ for any  $\x \in \Dcal_i^{(\zeta)}$. 
 Moreover, since any $\y \in \Dcal_i$ satisfies $y_i -y_{i+1}>0$, we can bound the last sum from below restricting it 
 on $\Dcal_{i}^{(\zeta)}$: 
 \begin{eqnarray}
 \sum_{\y \in \Dcal_i : \ \y-\h_n \in \Lcal'} \phi (\y)  e^{\xi} 
  (y_i-y_{i+1}) &\geq & 
  \sum_{\x \in \Dcal_i^{(\zeta)} : \ \y-\h_n \in \Lcal'} \phi (\y)  e^{\xi} 
  (y_i-y_{i+1}) \nonumber \\ 
  &\geq & (\Xi \zeta ) \cdot  \sum_{\x \in \Dcal_i^{(\zeta)} : \ \x-\h_n \in \Lcal'} \phi (\y). \nonumber %\label{eq:sum_low}
  \end{eqnarray}
  We use this lower bound in~\eqref{eq:sum_diff} and conclude that 
\begin{eqnarray} 
\lefteqn{\sum_{\y \in \Dcal_i : \ \y-\h_n \in \Lcal'} \phi (\y - \h_n) -  \sum_{\y \in \Dcal_{i+1} : \ \y-\h_n \in \Lcal'} \phi (\y - \h_n)\geq} \nonumber \\
& &(1+o(1)) \frac{k_0}{1-n^*(v)/n_v^*} \cdot p^{1/2} (c_i - c_{i+1}) \cdot (\Xi \zeta) \cdot \sum_{\y \in \Dcal_i^{(\zeta)} : \ \y-\h_n \in \Lcal'} \phi (\y).
\nonumber
%\label{eq:sum_diff_low}
\end{eqnarray}
We will approximate the last sum (scaled by $n^*(v)^{- \frac{k_0-1}{2}}$) by an integral of the multivariate gaussian pdf over $\Dcal_i^{(\zeta)}$. In particular, we will show the following. 
\begin{claim} \label{clm:sum_lower}  If $n^*(v)$ is such that $\Ncal_v^*$ is realised and $n^*$ is such that 
$\Pcal_0$ is realised, then
$$ \liminf_{n \to \infty} \frac{1}{n^*(v)^{\frac{k_0-1}{2}}}\cdot \sum_{\y \in \Dcal_i^{(\zeta)} : \ \y-\h_n \in \Lcal'} \phi (\y)>0.$$
\end{claim}
This concludes the proof of the lemma. 

\end{proof}
The proof of the Lemma~\ref{lem:int_decomp} together with~\eqref{eq:prob_i} yield that for $n$ sufficiently large and $i=1,\ldots, k_0-1$: 
\begin{equation*} 
\prb{v \in \Scal \Scal_1^{(i)} \mid  n^*(v), n_v^*, \cv^*} - \prb{v \in \Scal \Scal_1^{(i+1)} \mid n^*(v), n_v^*, \cv^*} >\Lambda \cdot (c_i- c_{i+1}) p^{1/2},
\end{equation*}
under the assumption that $c_1 > c_2 > \cdots > c_{k_0}$.
\end{proof}

\begin{proof}[Proof of Claim~\ref{clm:sum_lower}]

We argue by considering an approximation of the sum of $\phi(\y)$, for $\y\in \Dcal_{i}^{(\zeta)}$ such that 
 $\y -\h_n \in \mathcal{L'}$, by an appropriate Gaussian integral. For any $\x= (x_1,\ldots, x_{k_0-1})$, let $\Bcal_{\boldsymbol{\boldsymbol{x}}}$ denote the box 
$\prod_{j=1}^{k_0-1} [x_i, x_i+n^*(v)^{-1/2})$. 

\begin{claim}\label{lem:MVTEstimate}
For any $\boldsymbol{x}', \boldsymbol{x}'' \in \Bcal_{\boldsymbol{x}}$ we have 
$$\lvert \phi(\boldsymbol{x}') - \phi(\boldsymbol{x}'') \rvert = O\left(\frac{1}{\sqrt{n^*(v)}}\right),$$
uniformly over the choices of $\boldsymbol{x}, \boldsymbol{x}'$ and $\boldsymbol{x}''$.
\end{claim}

\begin{proof}
Suppose that $\boldsymbol{x}' = (x'_1, \ldots ,x'_{k_0-1})^T$ and $\boldsymbol{x}'' = (x''_1,\ldots, x''_{k_0-1})^T,$ both belonging to $\mathcal{B}_{\boldsymbol{x}}$. 
We appeal to the multidimensional mean value theorem: there exists some $c \in (0,1)$ such that,
\[
\lvert \phi(\x') - \phi(\x'') \rvert \leq \lVert \nabla \phi(\left(1-c)\x' + c \x'' \right) \rVert_{2} \lVert \x' - \x'' \rVert_{2}.
\]
For $i=1,\ldots,k_0-1$, we set $z_i = x'_i (1-c) + x''_i c$ and consider the order of $\lVert \nabla \phi((z_1,\ldots z_{k_0-1})) \rVert_{2}.$ We remark that the exponent in $\phi((z_1,\ldots z_{k_0-1}))$ is a quadratic form in $\boldsymbol{z} = (z_1,\ldots z_{k_0-1})^T$. We set
\[
-\frac{1}{2}\boldsymbol{z}^T \Sigma \boldsymbol{z} =: h(\boldsymbol{z}).
\]
%where $h$ quadratic in $z_0$ and $z_1.$ 
Hence it follows from the definition of $\phi$ that
\[
\lVert \nabla \phi(\boldsymbol{z}) \rVert_{2} = \lvert \phi(\boldsymbol{z}) \rvert \lVert \nabla h(\boldsymbol{z}) \rVert_{2}.
\]
Now we observe that as $h$ is quadratic in $z_1,\ldots, z_{k_0-1}$ it follows that,
\[\lVert \nabla h(\boldsymbol{z}) \rVert_{2} = O(\left| \boldsymbol{z} \right|_{\infty}^2) = O(1).
\]
We also have that $\phi(\boldsymbol{z}) \leq 1$. Therefore, 
\[
\lvert \phi(\boldsymbol{x}') - \phi(\boldsymbol{x''}) \rvert = O(\lVert \boldsymbol{x}' - \boldsymbol{x}''\rVert_{2}).
\]
However, as $\boldsymbol{x}'$ and $\boldsymbol{x}''\in \mathcal{B}_{\boldsymbol{x}}$ then we have that $\lVert \boldsymbol{x}' - \boldsymbol{x}''\rVert_{2} = O\left(1/\sqrt{n^*(v)}\right),$ which concludes the proof of this claim.
\end{proof}

Thus, applying the above, we can consider an approximation of the above sums by a Gaussian integral. Recall that 
$\mathcal{B}^{\infty}_{\log{n}} = \{\boldsymbol{z} = (z_1,\ldots, z_{k_0-1}) \in \mathbb{R}^{k_0-1} : \lVert \boldsymbol{z} \rVert_\infty < \log{n} \}$.

%\begin{lemma}\label{lem:RiemannAprox}
%
%Suppose $\mathcal{D} \subset \mathcal{B}^{\infty}_{\log{n}}$ is a Borel set. Then we have
%
%\[
%\left\lvert \frac{1}{n(v)^{\frac{k_0-1}{2}}} \sum_{\boldsymbol{x} \in \mathcal{D} \cap \Lcal'} \phi(\boldsymbol{x}) - \int_{\mathcal{D}} \phi(\boldsymbol{z}) \mathrm{d}\boldsymbol{z} \right\rvert = o(p^{1/2}).
%\]
%\end{lemma}
%
%
%\begin{proof}
% 
 For each $\x \in \Lcal'$ such that $\x + \h_n \in \Dcal_i^{(\zeta)}$,  let us consider
\begin{equation*}
\boldsymbol{z}_{\textrm{max}}(\boldsymbol{x}) \in  \argmax_{\boldsymbol{z} \in \Bcal_{\boldsymbol{x}}} \{ \phi(\boldsymbol{z}) \},
% \ \ \textrm{and}, \  \boldsymbol{z}_\textrm{min}(\boldsymbol{x}) = \min_{\boldsymbol{z} \in \Bcal_{\boldsymbol{x}}} \{ \phi(\boldsymbol{z}) \},
\end{equation*}
a maximiser of $\phi$ inside $\Bcal_{\boldsymbol{x}}$. 

%Thus we will now deduce upper and lower bounds in terms of the Gaussian integral over $\mathcal{D}.$ By Lemma \ref{lem:MVTEstimate} we have,
%
%\begin{align*}
%     \frac{1}{n(v)^{\frac{k-1}{2}}} \cdot 
% \sum_{\boldsymbol{x} \in \mathcal{D} \cap \Lcal'}
% \phi (\boldsymbol{x}) &=
% \frac{1}{n(v)^{\frac{k-1}{2}}} \cdot 
% \sum_{\boldsymbol{x} \in \mathcal{D} \cap \Lcal'}
% \phi (\boldsymbol{z}_\textrm{min}(\boldsymbol{x}))  \\ 
% & \hspace{2cm}+ \frac{1}{n(v)^{\frac{k-1}{2}}} \cdot 
% \sum_{\boldsymbol{x} \in \mathcal{D} \cap \Lcal'}
% \lvert \phi (\boldsymbol{x}) -\phi (\boldsymbol{z}_\textrm{min}(\boldsymbol{x})) \rvert \\
% &= \frac{1}{n(v)^{\frac{k-1}{2}}}\sum_{\boldsymbol{x} \in \mathcal{D} \cap \Lcal'}
% \phi (\boldsymbol{z}_\textrm{min}(\boldsymbol{x})) + O\left(\frac{1}{n(v)^{k/2}} \lvert \mathcal{D} \cap \Lcal' \rvert\right).
%\end{align*}
%We observe that $\lvert \mathcal{D} \cap \Lcal' | \leq |\Lcal'| = O(n(v)^{(k-1)/2}\log^{2}{n}).$ Therefore we observe that as $p \gg n^{-1/2}\log^{3}{n}$ it follows that
%$np \gg p$ and, hence
%\[
%\frac{\lvert \mathcal{D} \cap \Lcal' |}{n(v)^{k/2}} = O\left(\frac{\log^2{n}}{\sqrt{n(v)}}\right) = o(p^{1/2}).
%\]
%Thus by taking an appropriate bound using a Riemann integral across $\Bcal_{\boldsymbol{x}}$ we have,
%\[
%\frac{1}{n(v)} \cdot 
% \sum_{\boldsymbol{x} \in \mathcal{D} \cap \Lcal'}
% \phi (\boldsymbol{x}) \leq \int_{\mathcal{D}} \phi(\boldsymbol{x}) d\boldsymbol{x} + o(p^{1/2}).
% \]
We write: 
\begin{align*}
  &   \frac{1}{n^*(v)^{\frac{k_0-1}{2}}} \cdot 
 \sum_{\y \in \mathcal{D}_i^{(\zeta)}: \y - \h_n \in \Lcal'}
 \phi (\y) =
 \frac{1}{n^*(v)^{\frac{k_0-1}{2}}} \cdot 
 \sum_{\y \in \mathcal{D}_i^{(\zeta)}: \y - \h_n \in \Lcal'}
 \phi (\boldsymbol{z}_\textrm{max}(\y)) \\ 
 &\hspace{2cm}-\frac{1}{n^*(v)^{\frac{k_0-1}{2}}} \cdot 
 \sum_{\y \in \mathcal{D}_i^{(\zeta)}: \y - \h_n \in \Lcal'}
  \lvert\phi( \boldsymbol{z}_\textrm{max}(\y)) -\phi(\y)  \rvert.
% &= \frac{1}{n(v)^{\frac{k_0-1}{2}}}\sum_{\x \in \mathcal{D}_i^{(\zeta)}: \x - \h_n \in \Lcal'}
% \phi (\boldsymbol{z}_\textrm{max}(\boldsymbol{x})) - O\left(\frac{1}{n(v)^{k_0/2}} \lvert \Dcal_i^{(\zeta)} \cap \Lcal' \rvert\right) \\
% &\geq \int_{\Dcal_i^{(\zeta)} + \h_n} \phi(\x) \mathrm{d} \x - o(p^{1/2}). 
\end{align*}
The first sum is asymptotically bounded from below by an integral:
\begin{equation*}
\liminf_{n\to \infty} 
\frac{1}{n^*(v)^{\frac{k_0-1}{2}}}\sum_{\y \in \mathcal{D}_i^{(\zeta)}: \y - \h_n \in \Lcal'}
 \phi (\boldsymbol{z}_\textrm{max}(\boldsymbol{x})) \geq \int_{\Dcal_i^{(\zeta)}} \phi(\x) \mathrm{d} \x.
\end{equation*}
For the second term, let us observe that $|\{\y \in \mathcal{D}_i^{(\zeta)}: \y - \h_n \in \Lcal' \}|  \leq |\Lcal'| = O(n^*(v)^{(k_0-1)/2}\log^{2}{n})$. 
By Claim~\ref{lem:MVTEstimate}
$$  \lvert\phi( \boldsymbol{z}_\textrm{max}(\y)) -\phi(\y)  \rvert = O\left(n^*(v)^{-1/2}\right).
$$
Therefore,
\begin{eqnarray*}
\lefteqn{\frac{1}{n^*(v)^{\frac{k_0-1}{2}}} \cdot 
 \sum_{\y \in \mathcal{D}_i^{(\zeta)}: \y - \h_n \in \Lcal'}
  \lvert\phi( \boldsymbol{z}_\textrm{max}(\x)) -\phi(\x)  \rvert}  \\
&& \hspace{1cm}=
  O\left(n^*(v)^{-1/2 - (k_0-1)/2}\right)  \cdot |\{\y \in \mathcal{D}_i^{(\zeta)}: \y - \h_n \in \Lcal' \}| \\ 
&& \hspace{1cm}=  O(1) \cdot n^*(v)^{-1/2} \log^2 n. 
\end{eqnarray*}

As $p \geq n^{-1/2}\log^{3}n$, it follows that
$p^{1/2} \geq  (np)^{-1/2} \log^3 n$ and, since $\Ncal_v^*$ is realised,
\[
\frac{\log^2 n}{\sqrt{n^*(v)}}= O\left( \frac{\log^2 n}{(np)^{1/2}}\right) = o(p^{1/2}).
\]
We conclude that 
\begin{equation*}
\liminf_{n\to \infty}   \frac{1}{n^*(v)^{\frac{k_0-1}{2}}} \cdot 
 \sum_{\y \in \mathcal{D}_i^{(\zeta)}: \y - \h_n \in \Lcal'}
 \phi (\x)  \geq  \int_{\Dcal_i^{(\zeta)} } \phi(\y) \mathrm{d} \y - o(p^{1/2}).
\end{equation*}
Since $\limsup_{n\to\infty}\| \h_n \|_{\infty} < \infty$, we conclude that for some $\kappa >0$ we have  $\phi (\y)> \kappa$ for any 
$\y \in \Dcal_i^{(\zeta)}$. 
Therefore, 
$$ \liminf_{n\to\infty}  \int_{\Dcal_i^{(\zeta)}} \phi(\y) \mathrm{d} \y \geq \kappa \cdot \mathrm{Leb} (\Dcal_i^{(\zeta)}) >0,$$
and the claim follows.
%\end{proof}
\end{proof}

\subsection{Concentration after the first step} \label{sec:second_moment}
Let $X_i$ denote the random variable that is the number of vertices that obtain state $i \in \Mcal_0$ during round 1 in a strong sense. We will condition on $\cv^*$ which satisfies $\Ev_{n,\delta}$ and moreover 
$c_i - c_{i+1} >\delta$, for $i=1,\ldots, k_0-1$. 
Our aim in this section is to show that on this conditional space
 for any $\gamma >0$, a.a.s., for $i=1,\ldots, k_0-1$ 
\begin{equation}\label{eqn:X_0Bounds} 
  X_i - X_{i+1} \geq  \Omega (1) n p^{1/2} (c_{i+1}-c_i) - \gamma np^{1/2}.  
\end{equation}
We start with a lower bound on the conditional expectation of the difference $X_i - X_{i+1}$:
\begin{equation} \label{eq:X_0-exp}
\E{X_i - X_{i+1}| n^*, \cv^*} \geq  \Omega (1) n p^{1/2} (c_{i+1}-c_i).
\end{equation}

\begin{proof}[Proof of~\eqref{eq:X_0-exp}]
To see this, we write $X_i =\sum_{v\in V_n} \ind{v  \in \Scal\Scal_1^{(i)}}$. With this we can express:
\begin{eqnarray*} 
X_i - X_{i+1} = \sum_{v\in V_n} \ind{v  \in \Scal\Scal_1^{(i)}} - \sum_{v\in V_n} \ind{v  \in \Scal\Scal_1^{(i+1)}}.
\end{eqnarray*}
Thus, 
$$\E{X_i - X_{i+1}| n^*, \cv^*}  = \E{  \sum_{v\in V_n} \ind{v  \in \Scal\Scal_1^{(i)}} - \sum_{v\in V_n} \ind{v  \in \Scal\Scal_1^{(i+1)}} | n^*, \cv^*}.$$
Using the linearity of the conditional expectation, it is sufficient to show that  
$$ \E {\ind{v  \in \Scal\Scal_1^{(i)}} -  \ind{v  \in \Scal\Scal_1^{(i+1)}} | n^*, \cv^*} =\Omega (1) p^{1/2} (c_i -c_{i+1}).$$
We write 
\begin{eqnarray} 
&& \E {\ind{v  \in \Scal\Scal_1^{(i)}} -  \ind{v  \in \Scal\Scal_1^{(i+1)}} | n^*, \cv^*} = \nonumber \\
&&\hspace{1cm} \E{\E{\ind{v  \in \Scal\Scal_1^{(i)}} -  \ind{v  \in \Scal\Scal_1^{(i+1)}}| n^*(v)}| n^*, \cv^*} \nonumber \\ &\geq&  \E{\ind{\Ncal_v^*}\E{\ind{v  \in \Scal\Scal_1^{(i)}} -  \ind{v  \in \Scal\Scal_1^{(i+1)}}| n^*(v)}| n^*, \cv^*} 
\nonumber \\
&& \hspace{1.5cm} - \E{1-\ind{\Ncal_v^*} | n^*,\cv^*} \nonumber \\
&\stackrel{\eqref{eq:Nv_prob}}{=} &  \E{\ind{\Ncal_v^*}\E{\ind{v  \in \Scal\Scal_1^{(i)}} -  \ind{v  \in \Scal\Scal_1^{(i+1)}}| n^*(v)}| n^*, \cv^*} - o(1/n), \label{eq:cond_of_cond}
\end{eqnarray}
since $n^*(v)$ is independent of $\cv^*$ and therefore $\E{1-\ind{\Ncal_v^*} | n^*,\cv^*} = \E{1-\ind{\Ncal_v^*} | n^*}$.  
Now, by Lemma~\ref{lem:conclusion} we deduce that 
$$ \E{\ind{\Ncal_v^*}\E{\ind{v  \in \Scal\Scal_1^{(i)}} -  \ind{v  \in \Scal\Scal_1^{(i+1)}}| n^*(v)}| n^*, \cv^*}= \Omega (1) \cdot p^{1/2} (c_i - c_{i+1}). $$
Thus,~\eqref{eq:cond_of_cond} yields 
$$ \E {\ind{v  \in \Scal\Scal_1^{(i)}} -  \ind{v  \in \Scal\Scal_1^{(i+1)}} | n^*, \cv^*} =  \Omega (1) \cdot p^{1/2} (c_i - c_{i+1})$$
and~\eqref{eq:X_0-exp} follows. 
\end{proof}

%and, moreover, $q_{i+1} - q_i = \Omega (1) p^{1/2} (c_{i+1}-c_i)$, for $i=1,\ldots, k_0-1$. 
%%and consequently by \eqref{eq:betaBound},
%%\begin{equation} \label{eq:low_b}
% %   X_0 \geq \frac{n}{2} + \frac{\delta}{2} n p^{1/2}.  
%%\end{equation}
The next step towards the proof of~\eqref{eqn:X_0Bounds} is to show that 
\begin{equation}\label{eqn:X_i_conc} 
    \prb{|X_i - \E{X_i}| > \gamma n p^{1/2} \mid n^*, \cv^*} = o(1),
\end{equation}
Using this result with $\gamma/2$ instead of $\gamma$,~\eqref{eqn:X_0Bounds} follows from~\eqref{eq:X_0-exp}. 

\begin{proof}[Proof of~\eqref{eqn:X_i_conc}]
Chebyschev's inequality yields
\begin{equation*}
     \prb{|X_i - \E{X_i}| >  \gamma n p^{1/2}  \mid n^*, \cv^*} \leq \frac{1}{\gamma^2} \cdot \frac{\mathrm{Var} (X_i \mid n^*, \cv^*)}{n^2 p}. 
\end{equation*}
We will bound $\mathrm{Var}(X_i  \mid n^*, \cv^*)$ and, in particular, we will show that if 
$n^*$ satisfies $\Pcal_0$ and $\cv^*$ satisfies $\Ev_{\delta,n}$, then 
\begin{equation} \label{eq:var_bound}
    \mathrm{Var} (X_i \mid n^*, \cv^* ) = O\left( n^{3/2}\right). 
\end{equation}
(The proof of this is quite tedious and is postponed to the appendix.)
Now, this implies that on $\Pcal_0$ and $\Ev_{\delta, n}$
\begin{equation*}
  \prb{|X_i - \E{X_i}| >  \gamma n p^{1/2}  \mid n^*, \cv^*} = O\left( \frac{1}{n^{1/2}p}\right) = o(1). 
\end{equation*}
Hence, conditional on $n^*, \cv^*$ being such that $\Pcal_0$ and $\Ev_{\delta, n}$ occur, a.a.s.
\[
|X_i - \E{X_i}| < \gamma n p^{1/2}
\]
and thus~\eqref{eqn:X_i_conc} holds. 
\end{proof}

\section{After the first round} \label{sec:after_round_1}

In the previous subsection, we used a second moment argument to show that a.a.s. after the execution of the first round 
the system ends up with sets $S_1^{(1)},\ldots, S_1^{(k_0)}$ where $S_1^{(i)}$ consists of all vertices in $V_n$ that have 
adopted state $i$ in a strong sense during the execution of round 1, as well as with a subset $R_1 \subset V_n$ with $|R_1| = O (\log^2 n \sqrt{n/p})$ that consists of all those vertices that faced a tie during round 1 (cf. Claim~\ref{clm:residual_set}). 
Moreover, by~\eqref{eqn:X_0Bounds}, if $\cv^*$ is such that $\Ev_{\delta, n}$ is realised and 
 $c_1> \cdots > c_{k_0}$, then $|S_1^{(i)}| = b_i(n) $, where 
$b_i (n) - b_{i+1}(n)>\Omega (1)  (c_i-c_{i+1}) np^{1/2}$, for $i=1,\ldots, k_0-1$. Moreover, as 
$\Ev_{\delta, n}$ is realised, $c_i - c_{i+1} > \delta$.  
We prove the following. 
\begin{lemma} \label{lem:nbds_conc}
Let $\omega = \omega (n): \mathbb{N} \to \mathbb{R}$ be such that $\omega(n) \to \infty$ 
as $n\to \infty$ and $n^{1/3}p \geq \omega$. 
Suppose that  for $i=1,\ldots, k_0$, the functions $b_i : \mathbb{N} \to \mathbb{R}$ satisfy
$b_i (n) - b_{i+1}(n)>\Omega (1) np^{1/2}$, for $i=1,\ldots, k_0-1$.
For any $\delta >0$
a.a.s. the following holds: for any partition $(P_1,\ldots, P_{k_0}, R)$ of $V_n$ with 
$|P_i|= b_i (n)$ and $|R|=O (\log^2 n \sqrt{n/p})$, all, but fewer than 
$np/\omega^{1/2}$, vertices $v \in V_n$  have
$$d_{P_i}(v) - d_{P_{i+1}}(v)>  \frac{\delta}{4} n p^{3/2},$$
for all $i=1,\ldots, k_0-1$. 
\end{lemma}
\begin{proof}
We will use the union bound to show that a.a.s. there is no subset $S \subset V_n$ with
$|S| = np /\sqrt{\omega}$ and a partition $(P_1,\ldots, P_{k_0},R)$ with the property that for any $v\in S$ 
$d_{P_i}(v) - d_{P_{i+1}}(v)\leq  \frac{\delta}{4} n p^{3/2}$, for some $1\leq i \leq k_0-1$. 
In particular, we fix such a subset $S$ and a partition $(P_1,\ldots, P_{k_0},R)$ and we show that a.a.s. 
at least half of the vertices of $S$ (a subset $S'\subset S$) 
have degree inside $P_i$ that is close to $|P_i\setminus S|p$ for all 
$1\leq i \leq k_0-1$. This yields a lower bound on $d_{P_i\setminus S}(v) - d_{P_{i+1}\setminus S}(v)$ for all vertices in $S'$. The choice of $S'$ is such that all vertices inside $S'$ have degree inside $S$ that is proportional to $|S|p$. This together with the lower bound on $|P_{i+1}|-|P_i|$ ensures that for all $1\leq i \leq k_0-1$ the lower bound on $d_{P_i}(v) - d_{P_{i+1}}(v)$ holds. We now proceed with the details of this argument.

Let $S \subset V_n$ be such that 
$|S| = np /\sqrt{\omega}$ and let $(P_1,\ldots, P_{k_0},R)$ be a partition of $V_n$ as specified above. 
Firstly, we apply the Chernoff bound to deduce the concentration of the number of edges within $S$; we denote it by $e(S)$. 
We have \[
\mathbb{P}\left( \left|e(S) - \binom{|S|}{2} p\right| \geq \frac{p|S|^2}{\sqrt{\omega}} \right)  \leq \exp\left(- \Omega \left(\frac{p |S|^2}{\omega}\right)\right).
\]
As $p \geq \omega n^{-1/3}$, we deduce that $p|S|^2/\omega = \Omega(n \omega)$. Hence with probability at least $1 - 2^{-\Omega (\omega \cdot n)}$, we have that $e(S)$ is suitably concentrated about its expected value. Thereby, on this event, the average degree of $G(n,p)[S]$ is 
$$\frac{2e(S)}{|S|} \leq p |S| (1+o(1)).$$
So on the above event and for $n$ sufficiently large, there exists a subset $S' \subset S$
with $|S'|= |S|/2$ such that every vertex $v \in S'$ has $d_{S} (v) \leq 3 |S|p$.

Let us set $\alpha (n) = np^{1/2}$ and consider $P_i$. We will first bound the probability that a given $v \in S'$ has  
$$\left| d_{P_i\setminus S}(v) - |P_i\setminus S|p \right| > \frac{\eps}{2} \alpha (n)p,$$
where $\eps>0$ is some constant. 
To this end we can apply the Chernoff bound since $d_{P_i\setminus S} (v)\sim \bin (|P_i\setminus S|, p)$. 
Inequality~\eqref{eq:Chernoff} yields: 
\begin{equation*}
    \prb{ \left| d_{P_i\setminus S}(v) - |P_i\setminus S|p \right| > \frac{\eps}{2} \alpha (n) p} \stackrel{|P_i\setminus S|\leq n}{=} 
    \exp \left(-\Omega \left(\frac{\alpha (n) ^2 p^2}{np}\right)\right).
\end{equation*}
Note that by the choice of $\alpha (n)$, the quantity in the exponent on the right-hand side is 
$$\frac{\alpha (n) ^2 p^2}{np} = \frac{n^2 p^{3}}{np} = n p^{2}.$$ 
Moreover, $np^{2} \to \infty$ as $n\to \infty$, since we have assumed that $p\gg n^{-1/3}$. 

Now, the union bound implies that for any $v\in S'$ and any fixed $\eps>0$ with probability 
$1- \exp (-\Omega (np^2))$, we have for $1\leq i < k_0$ 
$$ \left| d_{P_i\setminus S}(v) - |P_i\setminus S|p \right| \leq \frac{\eps}{2} \alpha (n) p. $$
On this event, 
\begin{eqnarray*}
d_{P_i\setminus S} (v) - d_{P_{i+1}\setminus S}(v) &\geq& \left( |P_i\setminus S|- |P_{i+1}\setminus S| \right)p - 
\eps \alpha (n)p \\
&\geq&\left( b_i(n) - b_{i+1} (n) - 2|S|\right)p -\eps \alpha (n)p.
\end{eqnarray*}

Now, we write \begin{eqnarray*} 
d_{P_i}(v) - d_{P_{i+1}} (v) &\geq&  d_{P_i\setminus S} (v) - d_{P_{i+1}\setminus S}(v) - d_{P_{i+1}\cap S}(v) \\
&\geq & d_{P_i\setminus S} (v) - d_{P_{i+1}\setminus S}(v) - d_{S}(v) \\
&\geq&  d_{P_i\setminus S} (v) - d_{P_{i+1}\setminus S}(v) -3|S|p \\
&\geq& \left( b_i(n) - b_{i+1}(n) - 5|S|\right)p -\eps \alpha (n)p.
\end{eqnarray*}
But for some $\delta>0$, 
we have $b_i (n) - b_{i+1}(n) \geq \delta n p^{1/2}$, for any $n$ sufficiently large, whereas $|S| =o(np)$. 
Therefore, the right-hand side of the above is for $n$ sufficiently large 
$$  \left( b_i(n) - b_{i+1} (n) - 5|S|\right) p -\eps \alpha (n)p >\frac{\delta}{2} np^{3/2} - \eps np^{3/2} > \frac{\delta}{4} np^{3/2},$$
if $\eps < \delta/4$. Thus, we have shown that if $\eps < \delta /4$ and $n$ is sufficiently large, then 
$$ d_{P_i} (v) - d_{P_{i+1}}(v) > \frac{\delta}{4} np^{3/2}.$$

%If $|d_{P_i}(v) - |P_i|p| > \eps b_i(n) p$, for any $v\in S$, then of course this is also the case for any $v \in S'$. 
%Since $|P_i| \gg |S|$ and, moreover, 
%$b_i (n) = \Omega (np^{1/2}) \gg |S|$, it follows that for any $v\in S'$
%$$|d_{P_i\setminus S}(v) - |P_i\setminus S|p| > \eps b_i (n) p - 3|S|p> \frac{\eps}{2} a(n)p,$$
%provided that $n$ is sufficiently large.
Note that the family $\{d_{P_i\setminus S}(v) \}_{v \in S', i=1,\dots,k_0}$ consists of i.i.d. random variables which are binomially distributed with parameters $|P_i\setminus S|, p$. 
The Chernoff bound~\eqref{eq:Chernoff} together with the union bound implies that for any $v \in S'$
\begin{equation*}
    \prb{\mbox{for some $1\leq i \leq k_0$}\ |d_{P_i\setminus S}(v) - |P_i\setminus S|p| > \frac{\eps}{2} \alpha (n)p} = \exp \left(-\Omega \left(\frac{\alpha (n)^2 p^2}{np}\right)\right).
\end{equation*}
Since the family $\{d_{P_i\setminus S}(v) \}_{v \in S', i=1,\dots,k_0}$ is independent we then deduce that 
\begin{equation*}
    \prb{\mbox{$\forall v\in S' $  $\exists 1\leq i \leq k_0$:}\ |d_{P_i\setminus S}(v) - |P_i\setminus S|p| > \frac{\eps}{2} \alpha (n)p} = \exp \left(-\Omega \left(\frac{|S'|\alpha (n)^2 p^2}{np}\right)\right).
\end{equation*}
But since $|S'| = |S|/2$ and $|S|=np/\sqrt{\omega}$ we have
\begin{equation*}
    \frac{|S'|\alpha (n)^2 p^2}{np}=\Omega \left(\frac{ \alpha^2 (n) p^2}{\sqrt{\omega}} \right). 
\end{equation*}
As $\alpha (n) = np^{1/2}$, the numerator on the right hand side is
$$\alpha(n)^2  p^2 = \Omega (n^2 p^3).$$
Hence as $p \geq \omega n^{-1/3}$ we have that $\alpha^2 (n) p^2/\sqrt{\omega} \geq n\omega^{5/2}$, 
and, therefore, this probability is $e^{-\Omega (\omega^2 n)}$. 
%If instead we consider $d_{P_1}(v)$, the calculations remain almost identical; though we remark that as $|P_0| \leq n - np/\omega^{1/3}$ we have that $|P_1| \geq np/\omega^{1/3} \gg |S|.$ Thus the calculation of the appropriate probability follows with:
%\begin{equation*}
%    \prb{\mbox{for all $v \in S'$}\ |d_{P_1\setminus S}(v) - |P_1\setminus S|p| > \frac{\eps}{2} a(n)p} = \exp \left(-\Omega \left( \frac{\omega|S'|\alpha (n)^2 p^2}{np}\right)\right) = \exp\left(-\Omega(n\omega^{7/2})\right).
%\end{equation*}
As there are $2^{O(n)}$ choices for $S$and $S' \subset S$ and the partition $(P_1, \ldots, P_{k_0}, R)$, the union bound implies the lemma.
\end{proof} 
\begin{remark}
Note that it is in the proof of Lemma~\ref{lem:nbds_conc} that the assumption $p\gg n^{-1/3}$ is needed. This ensures that the bad events occur with probability at most $2^{-\omega \cdot n}$ and, therefore makes the union bound work. 
\end{remark}

We now apply Lemma~\ref{lem:nbds_conc} to deduce the evolution of the behaviour of all but at most $np/\omega^{1/2}$ vertices from round 1 into round 2. In particular, as $np^{1/2} \gg \log n\sqrt{n/p}$, the above 
lemma implies that all but at most $np/\omega^{1/2}$ will adopt state 1 after round 2. 
Finally, a union bound over all partitions of $V_n$ into two parts where one of the parts has size at most $n/\omega^{1/2}$ and the Chernoff bound imply that conditional on the above evolution of the process 
after the third round a.a.s. all vertices will have adopted state 1. This concludes the proof of Theorem~\ref{thm:main}. 

\section{Conclusions and some further directions}
In this paper we generalised the results of~\cite{ar:BerkDev_SPA,ar:CKLT2021,ar:FMK_RSA_2021} showing that 
majority dynamics on the random graph $G(n,p)$ and $k>2$ states will reach unanimity after at most 3 rounds, 
provided that $np\gg n^{2/3}$. As in the case of 2 states, the proof relies on a tight analysis of the configuration after the first round. 

We do not believe that the lower bound imposed on $p$ is tight. It is needed in the proof of Lemma~\ref{lem:nbds_conc} for the union bound argument to work. It could be possible to extend Theorem~\ref{thm:main} to $p\gg n^{-1/2}$ or to even lower densities as in~\cite{ar:CKLT2021}.

The ``power-of-the-few'' question becomes more intriguing in the context of more than 2 states. 
If the initial configuration is random, then one of the states dominates over each one of the other states 
by $\sqrt{n}$. What would happen if we started the process from a configuration where two or more of the 
states had exactly the same size? Could one extend the results of Tran and Vu~\cite{ar:TranVu2020} and those 
of Sah and Sawhney~\cite{ar:SahSawney2020} in this context?

\appendix 

\section{The second moment calculation} \label{app:second_moment}
\begin{proof}[Proof of~\eqref{eq:var_bound}]
To keep the presentation simple, we will drop the conditioning from our notation. 
As it will be obvious from in our subsequent calculations, we work with a given value of $n^*$ 
and a given realisation of $S_0$ such that $\Ev_0\cap \Ev_{\delta, n}$ is realised. 
Recall that for $i \in \Mcal_0$
we let $\Scal\Scal_1^{(i)}$ denote the subset of vertices $v \in V_n$ such that $S_1(v)=i$ in a \emph{strong sense}, that is, 
$n_i(v) > n_j(v)$ for any $j\in \Mcal_0$ with $j \not = i$. 
By the linearity of the expected value: 
\begin{equation*}
    \E{X_i} = \sum_{v\in V_n} \prb{v \in \Scal \Scal_1^{(i)}},
\end{equation*}
and 
\begin{equation*}
    \mathrm{Var} (X_i) \leq \sum_{v \neq v'} \left(\prb{v,v' \in \Scal \Scal_1^{(i)}} - \prb{v \in \Scal \Scal_1^{(i)}}\prb{v' \in \Scal \Scal_1^{(i)}}\right) + \E{X_i}.
\end{equation*}
We will show that uniformly over all distinct $v, v' \in V_n$ we have 
\begin{equation} \label{eq:local_variance}
   \prb{v,v' \in \Scal \Scal_1^{(i)}} - \prb{v \in \Scal \Scal_1^{(i)}}\prb{v' \in \Scal \Scal_1^{(i)}}= O\left(\frac{1}{n^{1/2}}\right).
\end{equation}
Using the trivial bound that 
$\E{X_i} \leq n$ we obtain:
\begin{equation*}
\mathrm{Var} (X_i) =    O\left(\frac{n^2}{n^{1/2}} + n\right)   = O\left( n^{3/2}\right).
\end{equation*}
\end{proof}

\begin{proof}[Proof of~\eqref{eq:local_variance}]
For any $v\in V_n$ let $J_v$ denote the indicator
random variable that is equal to 1 precisely when $v \in \Scal \Scal_1^{(i)}$.
We will bound 
\begin{equation*}
   d(v,v'):= \E{J_v \cdot J_{v'}} - \E{J_v} \cdot \E{J_{v'}}.
\end{equation*}
Recall that $S^* = \cup_{j \in \Mcal_0} S_0^{-1} (j)$.
Firstly, let us observe that if $v,v' \not \in S^*$, then the random variables $J_v$ and $J_{v'}$ are independent. In other words,
if $v,v'\not \in S^*$, then $d(v,v')=0$. 
Thus, it suffices to bound the sum 
\begin{equation} 
\sum_{(v,v'): v\not = v', \ v' \in S^*} d(v,v'). 
\end{equation}
(This sum is taken over ordered pairs.)
Let us fix such a pair $(v,v')$ with  $v\not = v'$ and $v \in V_n$ but $v' \in S^*$. 
Recall that we set $n_{v,v'}^* = n^* - 1_{v\in S^*} - 1_{v' \in S^*}= n^*- 1_{v\in S^*}-1$. 
Furthermore, in Section~\ref{sec:LLT} we set $n^*_v = n^* -1_{v \in S^*}$; thus, $n^*_{v,v'} = n^*_v - 1$. 
 
We defined $n^*_{v,v'}(v) = |N_{S^*\setminus \{v,v'\}}(v)|$; hence $n^*_{v,v'}(v) \sim 
\bin (n_{v,v'}^*,p)$. We apply the law of total probability
conditioning on whether or not $v\sim v'$ as well as on the values of $n^*_{v,v'}(v)$ and $n^*_{v,v'}(v')$: 
\begin{eqnarray}
   && \E{J_v \cdot J_{v'}} = \nonumber \\
   &&\hspace{1cm} p\cdot \E{\E{J_v \cdot J_{v'} \mid n^*_{v,v'}(v),n^*_{v,v'}(v')}\mid v\sim v'} \nonumber \\ 
&& \hspace{1.5cm} + (1-p)\cdot \E{\E{J_v \cdot J_{v'} \mid n^*_{v,v'}(v), n^*_{v,v'}(v')}\mid v\not \sim v'}. \nonumber \\
   && \label{eq:conditioning}
\end{eqnarray}
We will start with the first summand on the right-hand side. 
Recall that we are conditioning 
on the initial configuration 
$\Scal_{\Mcal_0}=\left\{S^{-1}_0(1),\ldots ,S^{-1}_0(k_0)\right\}$ 
being such that 
$n_j:=|S^{-1}_0(j)| = n^*/k_0 + c_j \sqrt{n^*}$, for $j\in \Mcal_0$, where $\delta < |c_j|<1/\delta$.

Let us abbreviate $\Es{\cdot}= \E{\cdot \mid v \sim v'}$.
Using these abbreviations, we can write explicitly 
\begin{eqnarray}
\lefteqn{\Es{\E{J_v \cdot J_{v'} \mid n^*_{v,v'}(v), n^*_{v,v'}(v')}} =} \nonumber \\
&&\sum_{\hat{k}=0}^{n^*_{v,v'}} \sum_{\hat{k'}=0}^{n^*_{v,v'}} 
\Es{J_v \cdot J_{v'} \mid n^*_{v,v'}(v) =\hat{k}, n^*_{v,v'}(v') =\hat{k}'} \cdot 
\prb{ n^*_{v,v'}(v) =\hat{k}, n^*_{v,v'}(v') =\hat{k}'}. \nonumber \\
&& 
\end{eqnarray}
Note that conditional on $v \sim v'$ 
and on $n^*_{v,v'}(v) =\hat{k}, n^*_{v,v'}(v') =\hat{k}'$, the random variables $J_v$ 
and $J_{v'}$ are independent. 
%The random variables
%$\hat{n}(v),\hat{n}(v')$ are independent 
%too. 
Furthermore, $J_v$ depends only on the realisation of $N_G(v)$ inside $S^*$.
%$N(v) \setminus \{v,v' \}$ and $J_{v'}$ depends only on $N(v')\setminus \{v,v'\}$ , if we condition on $v \sim v'$. 

So each summand in the above double sum can be written as: 
\begin{eqnarray*}
\lefteqn{\Es{J_v \cdot J_{v'} \mid n^*_{v,v'}(v) =\hat{k}, n^*_{v,v'}(v') =\hat{k}'} \cdot 
\prb{n^*_{v,v'}(v) =\hat{k}, n^*_{v,v'}(v') =\hat{k}'}=} \\
&& 
\Es{J_v  \mid n^*_{v,v'}(v) =\hat{k}, n^*_{v,v'}(v') =\hat{k}'} \cdot
\Es{J_{v'} \mid n^*_{v,v'}(v) =\hat{k}, n^*_{v,v'}(v') =\hat{k}'} 
\times \\
& &\hspace{1cm} \prb{ n^*_{v,v'}(v) =\hat{k}}\prb{ n^*_{v,v'}(v') =\hat{k}'} \\
&&=
\Es{J_v  \mid n^*_{v,v'}(v) =\hat{k}} \cdot
\Es{J_{v'} \mid n^*_{v,v'}(v') =\hat{k}'} 
\cdot 
\prb{ n^*_{v,v'}(v) =\hat{k}}\prb{ n^*_{v,v'}(v') =\hat{k}'}.
\end{eqnarray*}
Therefore, the double sum itself can be factorised as follows:
\begin{eqnarray}
\lefteqn{\Es{\E{J_v \cdot J_{v'} \mid n^*_{v,v'}(v), n^*_{v,v'}(v')}} =} \nonumber \\
&=&\left(\sum_{\hat{k}=0}^{n^*_{v,v'}} \Es{J_v  \mid n^*_{v,v'}(v) =\hat{k}} \cdot
\prb{\hat{n}^*_{v,v'}(v) =\hat{k}} \right)\times \nonumber \\
&& \hspace{2cm} \left(\sum_{\hat{k}'=0}^{n^*_{v,v'}} \Es{J_{v'}  \mid n^*_{v,v'}(v') =\hat{k}'} \cdot
\prb{ n^*_{v,v'}(v') =\hat{k}'} \right). \nonumber \\
&& \label{eq:factorised}
\end{eqnarray}
Let us fix $\hat{k} \in \{0,\ldots, n^*_{v,v'}\}$ and define 
$$\hat{\Kcal}^{(i)}_{\hat{k}} = \{(\hat{k}_1,\ldots ,\hat{k}_{k_0})\in \mathbb{N}_0^{k_0} \ : \ 
\sum_{j=1}^{k_0} \hat{k}_j = \hat{k}, \ \hat{k}_i > \hat{k}_j, \mbox{for all $j \not = i$}\}.$$
Using this, with $n_{j,v'}(v) = |N_{S_0^{-1}(j) \setminus \{v'\}} (v)|$ for $j \in \Mcal_0$, we express: 
\begin{eqnarray}
&&\Es{J_v  \mid n^*_{v,v'}(v) =\hat{k}} = 
\sum_{(\hat{k}_1,\ldots,\hat{k}_{k_0}) \in \hat{\Kcal}_{\hat{k}}^{(i)}} \prb{n_{j,v'}(v) = \hat{k}_j, \ j\in \Mcal_0 \mid n^*_{v,v'}(v) = \hat{k}}. \nonumber \\
&& \label{eq:prob_sum}
\end{eqnarray}
Now, set $k_j =\hat{k}_j+ 1_{S_0(v')=j}$, for $j\in \Mcal_0$. 
 We compare 
$\prb{n_{j,v'}(v) = \hat{k}_j, \ j\in \Mcal_0 \mid n^*_{v,v'}(v)=\hat{k}}$ with 
$\prb{n_j(v) = k_j, \ j\in \Mcal_0 \mid n^*(v)=\hat{k}+1}$. 
Let $\Kcal_{\hat{k}}^{(i)}= \{ (k_1,\ldots,k_{k_0}) \in \mathbb{N}_0^{k_0} \ : \ \sum_{j=1}^{k_0} k_j = \hat{k}+1, \ k_i > k_j \ \mbox{for all $j\not =i$} \}$.
\begin{claim} \label{clm:prob_ratio}
For any $\hat{k} \in \{0,\ldots, n^*_{v,v'} \}$ and all $(k_1,\ldots,k_{k_0}) \in \Kcal_{\hat{k}}^{(i)}$ 
the following holds: 
\begin{eqnarray*}
\lefteqn{\frac{\prb{n_{j,v'}(v) = \hat{k}_j, \ j\in \Mcal_0\mid n^*_{v,v'}(v)=\hat{k}}}{\prb{n_j(v) = k_j, \ j\in \Mcal_0 \mid 
n^*(v) =\hat{k}+1}}=}\\
&&
\frac{\prod_{j\in \Mcal_0} \binom{n_j-1_{S_0(v)=j}-1_{S_0(v')=j}}{k_j-1_{S_0(v')=j}}} {\binom{n^*_{v,v'}}{\hat{k}}} \cdot \left(\frac{\prod_{i \in\Mcal_0}\binom{n_j-1_{S_0(v)=j}}{k_j}}{\binom{n^*_v}{\hat{k}+1}} \right)^{-1} \\
&=& \frac{n^*-1_{S_0(v)\in \Mcal_0}}{\hat{k}+1} \cdot \frac{k_{s(v')}}{n_{S_0(v')}-1_{S_0(v')=S_0(v)}}.
\end{eqnarray*}
\end{claim}
\begin{proof}
Note that if $S_0(v')=j$
\begin{equation*}
    \binom{n_j-1_{S_0(v)=j}-1_{S_0(v')=j}}{k_j-1_{S_0(v')=j}} \cdot \binom{n_j-1_{S_0(v)=j}}{k_j}^{-1} = \frac{k_j }{n_j-1_{S_0(v)=j}}
\end{equation*}
but otherwise 
\begin{equation*}
    \binom{n_j-1_{S_0(v)=j}-1_{S_0(v')=j}}{k_j-1_{S_0(v')=j}} \cdot \binom{n_j-1_{S_0(v)=j}}{k_j}^{-1}=1.
    \end{equation*}
Similarly, since $v' \in S^*$
\begin{equation*}
    \binom{n^*_v}{\hat{k}+1} \cdot \binom{n_{v,v'}^*}{\hat{k}}^{-1} = \frac{n^*_v}{\hat{k}+1}.
\end{equation*}
%whereas if $s(v') \not \in \Mcal_0$, then 
%\begin{equation*}
%   \binom{n^*-1_{s(v)\in \Mcal_0}}{k+1_{s(v')\in \Mcal_0}} \cdot \binom{n_{v,v'}^*}{k}^{-1} = 1,
%\end{equation*}
%since $k+1_{s(v')\in \Mcal_0}=k$ and $n_{v,v'}^*= n^* -1_{s(v)\in \Mcal_0}-1_{s(v')\in \Mcal_0}= n^* - 1_{s(v)\in \Mcal_0}$.
So the claim follows as $n_v^*= n^* - 1_{S_0(v) \in \Mcal_0}$.
%, for precisely one $j\in \Mcal_0$, if $v \in S^*$, or for all $j \in \Mcal_0$, if $v \not \in S^*$. 
\end{proof}
The above claim implies that 
\begin{eqnarray*} 
\lefteqn{\prb{n_{j,v'}(v) = \hat{k}_j, \ j\in \Mcal_0 \mid n^*_{v,v'}(v) = \hat{k}} = }\\
&& \hspace{0.1cm}\frac{n^*_v}{\hat{k}+1} \cdot \frac{k_{S_0(v')}}{n_{S_0(v')}-1_{S_0(v')=S_0(v)}} \cdot\prb{n_j(v) = k_j, \ j\in \Mcal_0  \mid \hat{n}^*(v) = \hat{k}+1}.
\end{eqnarray*}
Now, let us relate $\prb{n^*_{v,v'}(v) =\hat{k}}$ 
with $\prb{n^*(v) =\hat{k}+1}$:
\begin{eqnarray*}
\prb{n^*_{v,v'}(v) =\hat{k}} &=& \binom{n^*_{v,v'}}{\hat{k}} p^{\hat{k}} (1-p)^{n^*_{v,v'}-\hat{k}} = \binom{n^*_{v}-1}{\hat{k}} p^{\hat{k}} (1-p)^{n^*_{v}-1-\hat{k}}\\
&=& \frac{1}{p}\cdot \frac{\hat{k} +1}{n^*_v}\cdot \binom{n^*_v}{\hat{k}+1} p^{\hat{k}+1} (1-p)^{n^*_v-(\hat{k}+1)} \\
&=& \frac{1}{p}\cdot \frac{\hat{k}+1}{n^*_v}\cdot
\prb{n^* (v) =\hat{k}+1}.
\end{eqnarray*}
With these we can write
\begin{eqnarray*}
\lefteqn{\sum_{\hat{k}=0}^{n^*_{v,v'}} \Es{J_v  \mid n^*_{v,v'}(v) =\hat{k}} \cdot
\prb{n^*_{v,v'}(v) =\hat{k}} =} \\
&& \sum_{\hat{k}=0}^{n^*_{v,v'}} \prb{n^*(v) =\hat{k}+1} \times \\
&& \sum_{(k_1,\ldots,k_{k_0}) \in \Kcal^{(i)}_{\hat{k}}} \frac{k_{S_0(v')}}{(n_{S_0(v')} -1_{S_0(v')=S_0(v)})p}\cdot \prb{n_j(v) = k_j, \ j\in \Mcal_0 \mid n^* (v) = \hat{k}+1}.
\end{eqnarray*}
We express 
$k_{S_0(v')} = (n_{S_0(v')} -1_{S_0(v')=S_0(v)})p + \delta_{s(v')}$. 
Using this, we get:
\begin{eqnarray*}
\lefteqn{\sum_{\hat{k}=0}^{n^*_{v,v'}} \Es{J_v  \mid n^*_{v,v'}(v) =\hat{k}} \cdot
\prb{n^*_{v,v'}(v) =\hat{k}} =} \\
&& \sum_{\hat{k}=0}^{n^*_{v,v'}}\sum_{(k_1,\ldots,k_{k_0}) \in \Kcal^{(i)}_{\hat{k}}}
 \prb{n^* (v) =\hat{k}+1}\prb{n_j(v) = k_j, \ j\in\Mcal_0 \mid \hat{n}^*(v) = \hat{k}+1}  \\
&&+ \sum_{\hat{k}=0}^{n^*_{v,v'}} \prb{\hat{n}^*(v) =\hat{k}+1} \times \\
&&\sum_{(k_1,\ldots,k_{k_0}) \in \Kcal^{(i)}_{\hat{k}}} \frac{\delta_{s(v')}}{(n_{S_0(v')} -1_{S_0(v')=S_0(v)})p} 
\prb{n_j(v) = k_j, \ j\in \Mcal_0 \mid n^*(v) = \hat{k}+1}.
\end{eqnarray*}
The first term is precisely equal to $\E{J_v},$ while, for the second term, we combine: 
%\begin{eqnarray*}
%\lefteqn{\sum_{\hat{k}=0}^{n-2} \prb{n(v) %=\hat{k}+1}\frac{\hat{k}+1}{(n-1)p}
%\E{J_v \mid n(v)= \hat{k}+1} =} \\
%&& 
%\sum_{\hat{k}=0}^{n-2} \prb{n(v) =\hat{k}+1}\left(1+ \frac{\delta_v}{(n-1)p}\right)
%\E{J_v \mid n(v)= \hat{k}+1} \\
%&\leq& \E{J_v} +  \sum_{\hat{k}=0}^{n-2} \prb{n(v) %=\hat{k}+1} \frac{|\hat{k}+1 - (n-1)p|}{(n-1)p}\\
%&=&\E{J_v}  + O\left(\frac{1}{(np)^{1/2}} \right).
%\$end{eqnarray*}
\begin{eqnarray*}
\lefteqn{\prb{n^*(v) =\hat{k}+1} \cdot \prb{n_j(v) = k_j, \ j\in \Mcal_0 \mid n^*(v) = \hat{k}+1} =}\\ && \prb{n_j(v) = k_j, \ j\in \Mcal_0}.\hspace{6cm}
\end{eqnarray*}
Note that the random variables $n_j(v)$
for $j\in \Mcal_0$ form an independent family and are all binomially distributed. 
Let $\Scal_{\hat{k}+1} = \{(k_1,\ldots,k_{k_0}) : \sum_{j=1}^{k_0} k_j = \hat{k}+1, 0\leq k_j\leq n^*_v \}$. We bound
\begin{align*}
&\left|\sum_{\hat{k}=0}^{n^*_{v,v'}} \sum_{(k_1,\ldots,k_{k_0}) \in \Kcal^{(i)}_{\hat{k}}} \frac{\delta_{S_0(v')}}{(n_{S_0(v')} -1_{S_0(v')=S_0(v)})p} 
\prb{n_j(v) = k_j, \ j\in \Mcal_0} \right|\leq \\
& \sum_{\hat{k}=0}^{n^*_{v,v'}} \sum_{(k_1,\ldots,k_{k_0}) \in \Kcal^{(i)}_{\hat{k}}} \frac{|\delta_{S_0(v')}|}{(n_{S_0(v')} -1_{S_0(v')=S_0(v)})p} 
\prb{n_j(v) = k_j, \ j\in \Mcal_0} \\
&\leq \sum_{\hat{k}=0}^{n^*_{v,v'}} \sum_{(k_1,\ldots,k_{k_0}) \in \Scal_{\hat{k}+1}}
\frac{|\delta_{S_0(v')}|}{(n_{S_0(v')} -1_{S_0(v')=s(v)})p} 
\prb{n_j(v) = k_j, \ j\in \Mcal_0} \\
&\leq  \sum_{\hat{k}=0}^{n^*_v} \sum_{(k_1,\ldots,k_{k_0}) \in \Scal_{\hat{k}+1}}
\frac{|\delta_{S_0(v')}|}{(n_{S_0(v')} -1_{S_0(v')=S_0(v)})p} 
\prb{n_j(v) = k_j, \ j\in \Mcal_0} \\
&\leq \sum_{k_1=0}^{n_1} \cdots \sum_{k_{k_0}=0}^{n_{k_0}}
\frac{|\delta_{S_0(v')}|}{(n_{S_0(v')} -1_{S_0(v')=S_0(v)})p} 
\prod_{j=1}^{k_0}\prb{n_j(v) = k_j} \\
&= \sum_{k_{S_0(v')}=0}^{n_{S_0(v')}}
\frac{|\delta_{S_0(v')}|}{(n_{S_0(v')} -1_{S_0(v')=S_0(v)})p} 
\prb{n_{S_0(v')}(v) = k_{S_0(v')}} \stackrel{n_{S_0(v')} = n^* p_{S_0(v')}}{=} O\left(\frac{1}{\sqrt{np}}\right).
\end{align*}
Therefore, 
\begin{eqnarray*}
\sum_{\hat{k}=0}^{n^*_{v,v'}} \E{J_v  \mid n^*_{v,v'}(v) =\hat{k}} \cdot
\prb{n^*_{v,v'}(v) =\hat{k}} \leq
 \Es{J_v} + O\left(\frac{1}{\sqrt{np}}\right).
\end{eqnarray*}
Using this upper bound in~\eqref{eq:factorised}, we get 
\begin{equation} \label{eq:factorized_1}
    \Es{\E{J_v \cdot J_{v'} \mid n^*_{v,v'}(v), n^*_{v,v'}(v')}}  \leq 
    \E{J_v} \cdot \E{J_{v'}} +  O\left(\frac{1}{\sqrt{np}}\right).
\end{equation}
Now, let us abbreviate $\Ens{\cdot}= \E{\cdot \mid v \not \sim v'}$.
Using the law of total probability and conditioning on the values 
of $n^*_{v,v'}(v)$ and $n^*_{v,v'}(v')$, we get
\begin{align}
&\Ens{\E{J_v \cdot J_{v'} \mid n^*_{v,v'}(v), n^*_{v,v'}(v')}} =\nonumber \\
& \sum_{\hat{k}=0}^{n^*_{v,v'}} \sum_{\hat{k'}=0}^{n^*_{v,v'}} 
\Ens{J_v \cdot J_{v'} \mid n^*_{v,v'}(v) =\hat{k}, n^*_{v,v'}(v') =\hat{k}'} \cdot 
\prb{n^*_{v,v'}(v) =\hat{k}, n^*_{v,v'}(v') =\hat{k}'}. \nonumber
\end{align}
We will bound  
$\Ens{J_v \cdot J_{v'} \mid n^*_{v,v'}(v)=\hat{k}, n^*_{v,v'}(v')=\hat{k}'}$. 
Note that 
\begin{equation*}
\begin{split}
   &\Ens{J_v \cdot J_{v'} \mid n^*_{v,v'}(v)=\hat{k}, n^*_{v,v'}(v')=\hat{k}'} =\\
   &\hspace{2cm} \Ens{J_v  \mid n^*_{v,v'}(v)=\hat{k}}
   \cdot 
   \Ens{J_{v'} \mid n^*_{v,v'}(v')=\hat{k}'}.
   \end{split}
\end{equation*}
We will deal with $ \Ens{J_v  \mid n^*_{v,v'}(v)=\hat{k}}$, as the analogous calculation can be carried out for $v'$. 
With $\Kcal^{(i)}_{\hat{k}}= \{(k_1,\ldots,k_{k_0}) \ : \ 
\sum_{j=1}^{k_0} k_j = \hat{k}, \ k_i > k_j \ \mbox{for all $j\not =i$}
\} $ we write
\begin{equation}
  \Ens{J_v  \mid n^*_{v,v'}(v)=\hat{k}} = 
  \sum_{(k_1,\ldots,k_{k_0}) \in\Kcal^{(i)}_{\hat{k}}} 
  \prb{n_{j,v'}(v) = k_j, \ j\in \Mcal_0 \mid n^*_{v,v'}(v) = \hat{k}}.
\end{equation}
We show the analogue of Claim~\ref{clm:prob_ratio}.
% relating 
%$\prb{\hat{n}_0(v;j) = k_j, \ j\in \Mcal_0 \mid \hat{n}^*_{v,v'}(v) = \hat{k}}$ with 
%$\prb{n_0(v;j) = k_j, \ j\in \Mcal_0 \mid \hat{n}^* (v) = \hat{k}}$. 
Recall that we are assuming that $s(v') \in \Mcal_0$.
\begin{claim}
For all $(k_1,\ldots,k_0) \in \Kcal^{(i)}_{\hat{k}}$ the following holds:
\begin{eqnarray*}
\lefteqn{\frac{\prb{n_{j,v'}(v) = k_j, \ j\in \Mcal_0 \mid n^*_{v,v'}(v) = \hat{k}}}{\prb{n_j(v) = k_j, \ j\in \Mcal_0 \mid n^*(v) = \hat{k}}}=}\\
\hspace{1cm}&&\frac{\prod_{j=1}^{k_0}\binom{n_j-1_{S_0(v)=j}-1_{S_0(v')=j}}{k_j}}{\binom{n^*_{v,v'}}{\hat{k}}} \cdot \left(\frac{\prod_{j=1}^{k_0}\binom{n_j-1_{S_0(v)=j}}{k_j}}{\binom{n^*_{v}}{\hat{k}}} \right)^{-1} \\
\hspace{1cm} &=&\frac{n^*_v}{n^*_v-\hat{k}} \cdot \frac{n_{S_0(v')} - k_{S_0(v')}-1_{S_0(v')=S_0(v)}}{n_{S_0(v')}-1_{S_0(v')=S_0(v)}}.
\end{eqnarray*}

\end{claim}

\begin{proof}
We follow a similar calculation to that in Claim \ref{clm:prob_ratio}.
We consider two cases. Firstly suppose $S_0(v) = S_0(v') = j$ for some $j\in \Mcal_0$. 
Then we have:

\[
\binom{n_j - 1_{S_0(v) = j} - 1_{S_0(v') = j}}{k_j} \cdot \binom{n_j - 1_{S_0(v) = j}}{k_j}^{-1} = \frac{n_j - k_j - 1}{n_j - 1},
\]
whereas for any $j' \neq j$ we have that:
\[
\binom{n_{j'} - 1_{S_0(v) = j'} - 1_{S_0(v') = j'}}{k_{j'}} \cdot \binom{n_{j'} - 1_{S_0(v) = j'}}{k_{j'}}^{-1} = 1.
\]

In the case where $s(v) = j$ but $S_0(v') = j'$ with $ j \neq j'$, the following hold:

\[\binom{n_j - 1_{S_0(v) = j} - 1_{S_0(v') = j}}{k_j} \cdot \binom{n_j - 1_{S_0(v) = j}}{k_j}^{-1}  = 1, 
\]
but
\[\binom{n_{j'} - 1_{S_0(v) = j'} - 1_{S_0(v') = j'}}{k_{j'}} \cdot \binom{n_{j'} - 1_{S_0(v) = j'}}{k_{j'}}^{-1}  = \binom{n_{j'} - 1}{k_{j'}} \cdot \binom{n_{j'}}{k_{j'}}^{-1} 
= \frac{n_{j'} - k_{j'}}{n_{j'}}.
\]
For any $j''\not = j',j$, we have 
\[\binom{n_{j''} - 1_{S_0(v) = j''} - 1_{S_0(v') = j''}}{k_{j''}} \cdot \binom{n_{j''} - 1_{S_0(v) = j''}}{k_{j''}}^{-1}  = 1.
\]

Furthermore, since $s(v')\in \Mcal_0$, then $n^*_{v,v'} = n^*_{v} -1$, whereby

\[
    \binom{n^*_{v}}{\hat{k}} \cdot \binom{n^*_{v,v'}}{\hat{k}}^{-1} = \frac{n^*_{v}}{n^*_v- \hat{k}}.
\]

\begin{comment}
Suppose now that $s(v') \not \in \Mcal_0$. 
In this case, for all $j \in \Mcal_0$
\[
\binom{n_{j} - 1_{s(v) = j} - 1_{s(v') = j}}{k_{j}} \cdot \binom{n_{j} - 1_{s(v) = j}}{k_{j}}^{-1} = 1,
\]
and
\[
    \binom{n^*_{v}}{\hat{k}} \cdot \binom{n^*_{v,v'}}{\hat{k}}^{-1} = 1,
\]
since $n_{v,v'}^* = n_v^*$.
\end{comment}
Thus, the claim follows.
\end{proof}
Therefore,
\begin{eqnarray} 
\lefteqn{\prb{\hat{n}_0(v;j) = \hat{k}_j, \ j\in \Mcal_0 \mid \hat{n}^*_{v,v'}(v) = \hat{k}} = }\nonumber \\
&& \frac{n^*_v}{n^*_v-\hat{k}} \cdot \frac{n_{S_0(v')} - k_{S_0(v')}-1_{S_0(v')=S_0(v)}}{n_{S_0(v')}-1_{S_0(v')=S_0(v)}}\cdot\prb{n_0(v;j) = k_j, \ j\in\Mcal_0 \mid \hat{n}^*(v) = \hat{k}}. \nonumber \\
&& \label{eq:transformation}
\end{eqnarray}
Furthermore, (since $v'\in \Mcal_0$, whereby $n^*_{v,v'}=n^*_{v} -1$)
\begin{eqnarray*}
\prb{n^*_{v,v'}(v) =\hat{k}} &=& \binom{n^*_{v,v'}}{\hat{k}} p^{\hat{k}} (1-p)^{n^*_{v,v'}-\hat{k}} \\
&=& \frac{1}{1-p}\cdot \frac{n^*_v-\hat{k}}{n^*_{v}}\cdot \binom{n^*_v}{\hat{k}} p^{\hat{k}} (1-p)^{n^*_v-\hat{k}} \\
&=& \frac{1}{1-p}\cdot \frac{n^*_v-\hat{k}}{n^*_v}\cdot
\prb{n^*(v) =\hat{k}}.
\end{eqnarray*}
Therefore,
\begin{align*}
&\prb{n^*_{v,v'}(v)=\hat{k}}\prb{n_{j,v'}(v) = \hat{k}_j, \ j\in \Mcal_0 \mid n^*_{v,v'}(v) = \hat{k}} =\\
&=\frac{n_{S_0(v')} - k_{S_0(v')}-1_{S_0(v')=S_0(v)}}{(n_{S_0(v')}-1_{S_0(v')=S_0(v)})(1-p)}\cdot\prb{n_j(v) = k_j, \ j\in \Mcal_0 \mid n(v) = \hat{k}}\prb{n^*(v)=\hat{k}}.
\end{align*}
Now, we express $k_{S_0(v')} =p(n_{S_0(v')}-1_{S_0(v)=S_0(v')}) + \delta_{S_0(v')}$, where $n_{S_0(v')} = n^* p_{S_0(v')}$. Using this, we get 
%\begin{eqnarray*}
%\frac{n_{s(v')} - k_{s(v')}-1_{s(v')=s(v)}}{n_{s(v')}-1_{s(v')=s(v)}} &=& 1-\frac{\hat{k}}{n-1} + 
%O \left( \frac{|\delta_{s(v')}|+1}{n}\right)\\
%& &  .
%\end{eqnarray*}
%Now, we write $n(v) = (n-1)p + \delta_v$ and so 
%\begin{equation*}
%    \frac{1}{1-p} \left( 1-\frac{\hat{k}}{n-1}\right) = 1- %\frac{\delta_v}{(n-1)(1-p)}
%\end{equation*}

%Thus, we conclude that 
%\begin{eqnarray} \label{eq:prob_ratio_asympt}
%\frac{n_{s(v')} - k_{s(v')}-1_{s(v')=s(v)}}{(n_{s(v')}-1_{s(v')=s(v)})(1-p)} = 1 - \frac{\delta_v}{(n-1)(1-p)}+ O \left( \frac{|\delta_{s(v')}|+1}{n}\right).
%\end{eqnarray}
%Using~\eqref{eq:prob_ratio_asympt} into~\eqref{eq:transformation} and summing over $\hat{k}$ and $(k_0,k_1,k_2) \in \Kcal_{\hat{k}}$ we get: 
\begin{align*}
&\sum_{\hat{k}=0}^{n^*_{v,v'}} \prb{n^*(v)=\hat{k}} \times \nonumber \\
& \sum_{(k_1,\ldots,k_{k_0}) \in \Kcal^{(i)}_{\hat{k}}}\frac{n_{S_0(v')} - k_{S_0(v')}-1_{S_0(v')=S_0(v)}}{(n_{S_0(v')}-1_{S_0(v')=S_0(v)})(1-p)}   \prb{n_j(v) = k_j, \ j\in \Mcal_0 \mid \hat{n}^*(v) = \hat{k}}= \nonumber \\
& \sum_{\hat{k}=0}^{n^*_{v,v'}}  \prb{n^*(v)=\hat{k}}  \times \nonumber \\
& \sum_{(k_1,\ldots ,k_{k_0}) \in \Kcal^{(i)}_{\hat{k}}}
\frac{(n_{S_0(v')}-1_{S_0(v')=S_0(v)})(1-p) + \delta_{S_0(v')}}{(n_{S_0(v')}-1_{S_0(v')=S_0(v)})(1-p)} \prb{n_j(v) = k_j, \ j\in\Mcal_0 \mid n^*(v) = \hat{k}}. \nonumber \\
%&& + \sum_{\hat{k}=0}^{n-2} \prb{n(v)=\hat{k}} \frac{\delta_v}{(n-1)(1-p)} %\sum_{(k_0,k_1,k_2) \in \Kcal_{\hat{k}}} \prb{n_0(v;j) = k_j, \ j=0,1,2 \mid %n(v) = \hat{k}} \nonumber \\
%&&+ O(1) \cdot  \sum_{\hat{k}=0}^{n-2} \prb{n(v)=\hat{k}} \sum_{(k_0,k_1,k_2) %\in \Kcal_{\hat{k}}} \frac{|\delta_{s(v')}|+1}{n} \prb{n_0(v;j) = k_j, \ %j=0,1,2 \mid n(v) = \hat{k}} \nonumber \\
%&& 
%&\label{eq:sum_expanded}
\end{align*}
As above, the first summand is
\begin{eqnarray}
\sum_{\hat{k}=0}^{n^*_{v,v'}} \prb{n^*(v)=\hat{k}} \sum_{(k_1,\ldots,k_{k_0}) \in \Kcal^{(i)}_{\hat{k}}} \prb{n_j(v) = k_j, \ j\in \Mcal_0 \mid n^* (v) = \hat{k}}  \leq \E{J_v},\nonumber \label{eq:term1}
\end{eqnarray}
and the second summand is bounded as
\[
\sum_{(k_1,\ldots,k_{k_0}) \in \Scal_{\hat{k}}}
\frac{ |\delta_{S_0(v')}|}{(n_{S_0(v')}-1_{S_0(v')=S_0(v)})(1-p)} \prb{n_j (v) = k_j, \ 
j\in \Mcal_0 \mid n^*(v) = \hat{k}} = O\left(\sqrt{\frac{p}{n}}\right).
\]
These two imply that 
\begin{eqnarray}
\E{\Ens{J_v\cdot J_{v'}  \mid \hat{n}(v), \hat{n}(v')}}
&=&\E{\Ens{J_v  \mid n^*_{v,v'}(v)}} \cdot \E{\Ens{J_{v'}  \mid n^*_{v,v'}(v')}} \nonumber \\
&\leq& \E{J_v}\cdot \E{J_{v'}} +  O\left(\sqrt{\frac{p}{n}}\right). \label{eq:unconnected}
\end{eqnarray}
Using the bounds of~\eqref{eq:factorized_1} and~\eqref{eq:unconnected}
into~\eqref{eq:conditioning} we deduce that 
\[\E{J_v \cdot J_{v'}} \leq \E{J_v}\cdot \E{J_{v'}} +  O\left(\frac{1}{n^{1/2}}\right).\]

\end{proof}

\end{document}